\documentclass{article}
\usepackage[utf8]{inputenc}
\usepackage{amsmath}
\usepackage{amssymb}
\usepackage{graphicx}
\usepackage{dsfont}
\usepackage{upgreek}
\usepackage{textcomp}
\usepackage{braket}
\usepackage[margin=1in]{geometry}
\usepackage{amsthm}
\usepackage{mathrsfs}
\usepackage{mathtools}
\usepackage{biblatex}
\def\bibfont{\footnotesize}
\numberwithin{equation}{section}
\newtheorem{theorem}{Theorem}[section]
\newtheorem{lemma}{Lemma}[section]
\newtheorem{corollary}{Corollary}[section]
\newtheorem{remark}{Remark}[section]

\newtheorem{proposition}{Proposition}[section]
\usepackage{scalerel,stackengine}
\stackMath
\newcommand\reallywidehat[1]{%
\savestack{\tmpbox}{\stretchto{%
  \scaleto{%
    \scalerel*[\widthof{\ensuremath{#1}}]{\kern-.6pt\bigwedge\kern-.6pt}%
    {\rule[-\textheight/2]{1ex}{\textheight}}%
  }{\textheight}%
}{0.5ex}}%
\stackon[1pt]{#1}{\tmpbox}%
}
\title{Multilinear Weighted Estimates and Quantum Zakharov System}
\author{BJ Choi}
\date{}
\begin{document}
\maketitle
\begin{abstract}
We consider the compact case of one-dimensional quantum Zakharov system, as an initial-value problem with periodic boundary conditions. We apply the Bourgain norm method to show low regularity local well-posedness for a certain class of Sobolev exponents that are sharp up to the boundary, under the condition that Schr\"odinger Sobolev regularity is non-negative. Using the conservation law and energy method, we show global well-posedness for sufficiently regular initial data, without any smallness assumption. Lastly we show the semi-classical limit as $\epsilon\rightarrow 0$ on a compact time interval, whereas the quantum perturbation proves to be singular on an infinite time interval.    
\end{abstract}
\section{Introduction.}\label{intro}
We consider the well-posedness and the semi-classical limit of the compact one-dimensional quantum Zakharov system (QZS).  Thus we assume the periodic boundary condition
\begin{equation}\label{eq:main}
      \begin{cases}
    (i\partial_t + \alpha \partial_{xx} - \epsilon^2\partial_{xx}^2)u = un,\:(x,t) \in \mathbb{T}\times [0,T]\\
    (\beta^{-2}\partial_{tt} - \partial_{xx} + \epsilon^2 \partial_{xx}^2)n = \partial_{xx}(|u|^2),\\
    (u(x,0),n(x,0),\partial_t n(x,0)) = (u_0, n_0, n_1) \in H^{s,l} \coloneqq H^s(\mathbb{T}) \times H^l (\mathbb{T}) \times H^{l-2}(\mathbb{T}),
    \end{cases}
\end{equation}
where $u$ is complex-valued, $n$ is real-valued, $\mathbb{T} = \frac{\mathbb{R}}{2\pi \mathbb{Z}}$, $T>0$ is the time-of-existence (to be determined), and $\alpha,\beta>0$, $s,l\in \mathbb{R}$. When $\epsilon=0$, QZS is well-known as the classical Zakharov system (ZS), a pair of non-linear PDE developed to model the interaction of Langmuir turbulence waves and ion-acoustic waves.  Here $u(x,t)$ denotes the slowly-varying envelope of electric field, and $n(x,t)$ represents an ion-acoustic wave  that models the density fluctuation of ions \cite{zakharov1972collapse}. A thrust of interest in rigorously studying the quantum effects unexplained by ZS came from the physics community \cite{garcia2005modified}.  There the quantum effect is characterized by a fourth-order perturbation with a quantum parameter $\epsilon>0$ that is non-negligible when either the ion-plasma frequency is high or the electrons temperature is low; for more background in the physics of this model, see \cite{haas2011quantum,simpson2009arrest}.

Our goal is to understand the effect of quantum perturbations, represented by the biharmonic operator.  We will do this in the context of well-posedness theory, thereby extending results of \cite{takaoka1999well}. We show that the biharmonic operator provides an extra degree of smoothing that nullifies the distinction between resonance ($\frac{\beta}{\alpha}\in \mathbb{Z}$) and non-resonance ($\frac{\beta}{\alpha}\notin \mathbb{Z}$), something which played a central role in \cite{takaoka1999well}.  
More precisely, we show that the regions of Sobolev exponent pairs $(s,l)\in\mathbb{R}^2$ yielding well-posedness for $\epsilon=0$ (which depend on $\frac{\beta}{\alpha}\in \mathbb{Z}$ or $\frac{\beta}{\alpha}\notin\mathbb{Z}$), are no longer different when $\epsilon>0$. We apply the Bourgain norm method to show that if ZS is well-posed in a certain Sobolev space of initial data, then so is QZS. Under the condition $s\geq 0$, we show that our application of Bourgain norm method yields a region of Sobolev exponents for the local well-posedness that is sharp up to the boundary. With the more precise statement given in Section \ref{nonlinearestimates}, we state our main result.  We define the region $\Omega_L\subseteq\mathbb{R}^2$ by
\begin{equation}\label{admissible sufficient}
    \Omega_L \coloneqq \left\{s\geq 0,\: -1\leq l <2s+1,\: -2<s-l\leq 2\right\}.
\end{equation}
\begin{theorem}\label{thm1}
Let $\alpha,\beta,\epsilon>0$ and $(s,l)\in \Omega_L$. 
Then for every $(u_0,n_0,n_1)\in H^{s,l}$, there exists $(u,n,\partial_t n) \in C([0,T],H^{s,l})$, a strong solution to (\ref{eq:main}), where $T = T(\lVert u_0\rVert_{H^s}, \lVert n_0\rVert_{H^l},\lVert n_1\rVert_{H^{l-2}})>0$. The solution is unique in the modified Bourgain space $X\subsetneq C([0,T],H^{s,l})$ and the data-to-solution map is Lipschitz continuous from $H^{s,l}$ to $X$.
\end{theorem}
In the appendix, we give examples of spacetime functions that illustrate the necessity of the condition
\begin{equation}\label{admissble necessary}
    s\geq -1,\: -1\leq l \leq 2s+1,\: -2\leq s-l\leq 2.  
\end{equation}
for control of the nonlinear terms in the appropriate Bourgain norms.

Although the QZS model is relatively new, the method of multilinear weighted estimates via Fourier transform and the Cauchy-Schwarz inequality has been used successfully by many.  These include (but are not limited to) Bourgain, Kenig-Ponce-Vega, and Ginibre-Tsutsumi-Velo \cite{bourgain1993fourier,kenig1996bilinear,kenig1996quadratic,ginibre1997cauchy}, in applications to various dispersive equations such as KdV, nonlinear Schr\"odinger equation with various nonlinearities, and ZS on $\mathbb{R}^d$. Additionally, Tao \cite{tao2001multilinear} investigated an alternative approach based on orthogonality and dyadic decompositions.

Typically the task of proving boundedness for certain multilinear operators reduces to spacetime Lebesgue-type estimates in Fourier space, which can be a challenge on periodic spatial domains where satisfactory Strichartz estimates are not available. Despite this difficulty, see \cite{erdougan2013smoothing,kishimoto2013local,bourgain1993cauchy} for various applications of Bourgain norm methods to ZS on periodic domains. On $\mathbb{R}^d$, as opposed to the compact case, it is generally expected that there is a wider range of Sobolev exponents for a well-posedness theory, with the full range of Strichartz estimates at one's disposal; for more recent work on QZS on $\mathbb{R}$, see \cite{fang2017local,chen2017low,jiang2014one}.

The QZS defines a Hamiltonian PDE with an energy functional $H$ defined on $H^{2,1}$; see Section \ref{back} for an explicit representation of this. We show, via the conservation law and an energy method, that the local flow obtained from Theorem \ref{thm1} is global whenever initial data are sufficiently regular, with finite energy. 
\begin{theorem}\label{thm2}
If $(u_0,n_0,n_1) \in \Omega_G\subseteq H^{s,l}$ with $\Omega_G\coloneqq\left\{0\leq s-l \leq 2,\:s+l \geq 4\right\}\cup \left\{(2,1)\right\}$, then the local solution obtained from Theorem \ref{thm1} can be extended to a global solution.
\end{theorem}
Here the difficulty is proving persistence of regularity, given that any initial data with a finite energy has a global solution in $H^{2,1}$, for which we derive an explicit growth rate of Sobolev norms. While our energy method for QZS provides an exponential bound on growth in time, see \cite{erdougan2013smoothing} for results on polynomial growth rates for the classical ZS on $\mathbb{T}$. 

We expect, however, that the above QZS local flow can be uniquely extended to a global flow, from scaling-invariance perspectives suggested in \cite{ginibre1997cauchy}, and provide here a heuristic argument for this. Assuming for the moment that  $\alpha=\beta=\epsilon=1$, suppose the long-time behavior of the solution is governed by the simplified system  
\begin{equation}
      \begin{cases}
    (i\partial_t  - \Delta^2)u = un,\\
    (\partial_{tt} +  \Delta^2)n = \Delta(|u|^2).
    \end{cases}
\end{equation}
Assuming further that both $n(x,t),\partial_t n(x,t)$ have mean zero for all $t\in\mathbb{R}$, consider a change of variable $N_{\pm} = n \pm i\Delta^{-1}\partial_t n$, under which from the previous equation yields
\begin{equation}\label{scale}
      \begin{cases}
    (i\partial_t  - \Delta^2)u = u\frac{N_+ + N_-}{2},\\
    (i\partial_t \mp \Delta)N_{\pm} = \mp |u|^2.
    \end{cases}
\end{equation}
If we add the assumption, as in \cite{ginibre1997cauchy}, that the higher order biharmonic operator dominates the scaling property of $N_{\pm}$, we can neglect the $\mp\Delta$ in the $N_{\pm}$ equation.  Then (\ref{scale}) is scale-invariant under
\begin{equation}
    u_\lambda (x,t) = \lambda^4 u(\lambda x, \lambda^4 t);\: n_\lambda (x,t)=\lambda^4 n(\lambda x, \lambda^4 t),
\end{equation}
and hence the pair of critical Sobolev exponents is
\begin{equation}
    (s_c,l_c) = (\frac{d}{2}-4,\frac{d}{2}-4).
\end{equation}
When $d=1$, $(s_c,l_c)$ is strictly below the region of well-posedness (see (\ref{admissible sufficient}), (\ref{admissble necessary})), and hence we expect any QZS local solution to be global. Moreover we remark that Theorem \ref{thm2} includes the case $\left\{s-l=0\right\}$, which is relevant from the scaling perspective; recall from \cite{ginibre1997cauchy} that for ZS, we have $(s_c,l_c) = (\frac{d}{2}-\frac{3}{2},\frac{d}{2}-2)$.

In the last part of the paper, we consider the semi-classical limit of QZS to ZS as $\epsilon\rightarrow 0$. Under the $\epsilon$-perturbation,  we expect the qualitative behavior of solutions to differ from that of the unperturbed system, and hence singular perturbation theory lies at the core of the analysis of QZS. As is well known, similar issues arise in the WKB method, multiscale analysis, and boundary layer theory; see \cite{deville2008analysis} for an application of singular perturbation theory to ODE in the context of renormalization group and normal form method. Here we extend the results of Guo-Zhang-Guo \cite{guo2013global} to show that the solutions behave continuously as $\epsilon\rightarrow 0$ on a compact time interval. Although their work is on $\mathbb{R}^d$ and for integer Sobolev exponents, an analogue of their argument works on $\mathbb{T}$ as well, and extends to non-integer exponents.  On the other hand, we provide a simple example that illustrates that the biharmonic operator $\epsilon^2 \Delta^2$, for any $\epsilon>0$, is a singular perturbation on an infinite time interval. Here we address a subtlety based on the fact that QZS generates a flow on $H^{s,l}$ whereas the classical ZS does so on $H_0^{s,l} \coloneqq H^s(\mathbb{T}) \times H^l (\mathbb{T}) \times H^{l-1}(\mathbb{T})$. To overcome this apparent \textit{discontinuity} of solution space, we need to uniformly bound the solution in various norms, with bounds  independent of $\epsilon>0$.

We briefly outline the organization of the paper. In Section \ref{back}, we introduce important notations and invoke the Lagrangian formalism of (\ref{eq:main}). In Section \ref{linearestimates}, we summarize a set of linear estimates that are used throughout the paper. In Section \ref{nonlinearestimates}, nonlinear estimates are proved and applied to yield local well-posedness of (\ref{eq:main}); in particular, we prove the more precise statement of theorem \ref{thm1}. In Section \ref{gwp}, we extend local solutions to global solutions for a fixed $\epsilon>0$ and consider the $\epsilon\rightarrow 0$ problem. Throughout the paper, $\alpha,\beta>0$ are fixed and the adiabatic limit $\beta\rightarrow \infty$ is not considered.
\section{Background.}\label{back}

As is conventional, we first define Fourier transform of $f \in L^2(\mathbb{T})$ and the inverse transform of $F \in l^2(\mathbb{Z})$:
\begin{align}
    \hat{f}(k) = \frac{1}{2\pi}\int f(x)e^{-ikx}dx;\:\Check{F}(x) &= \sum_{k \in \mathbb{Z}} F(k)e^{ikx}.
\end{align}

We use $\langle x \rangle = (1+x^2)^{\frac{1}{2}}$ and define a family of Sobolev spaces $W^{s,p}, \dot{W^{s,p}}$ (inhomogeneous and homogeneous, respectively) with $s\in \mathbb{R}, p \in (1,\infty)$ as
\begin{equation}
    \lVert f \rVert_{W^{s,p}} = \lVert \langle \nabla \rangle^s f \rVert_{L^p};\:\lVert f \rVert_{\dot{W}^{s,p}} = \lVert | k |^s \hat{f}(k)\rVert_{L^p},
\end{equation}
and of particular importance is when $p=2$ for which we write $H^s, \dot{H}^s$ as is usual, and the norms are defined via Fourier multipliers:
\begin{equation}
    \lVert f \rVert_{H^s} = \lVert \langle k \rangle^s \hat{f}(k)\rVert_{L^2};\: \lVert f \rVert_{\dot{H}^s} = \lVert | k |^s \hat{f}(k)\rVert_{L^2}.
\end{equation}
Whenever we take the direct sum of normed spaces, we will define the product norm to be the sum of the components, for instance, $\lVert (u_0,n_0,n_1)\rVert_{H^{s,l}}=\lVert u_0 \rVert_{H^s}+ \lVert n_0 \rVert_{H^l}+\lVert n_1 \rVert_{H^{l-2}}$.

As a consequence of the invariance of (\ref{eq:main}) under $u(x,t)\mapsto e^{i\theta}u(x,t)$ and time-translation, mass and energy are conserved:
\begin{align}\label{conservation}
    M[u,n,\partial_t n](t) &= \lVert u \rVert_{L^2}^2 = \lVert u_0 \rVert_{L^2}^2\nonumber\\
    H[u,n,\partial_t n](t) &= \alpha \lVert \partial_x u \rVert_{L^2}^2 + \epsilon^2 \lVert \partial_{xx}u \rVert_{L^2}^2 + \frac{1}{2}\Big(\lVert n \rVert_{L^2}^2 + \frac{1}{\beta^2}\lVert \partial_t n \rVert_{\dot{H}^{-1}}^2 + \epsilon^2 \lVert\partial_x n \rVert_{L^2}^2\Big) + \int n |u|^2.
\end{align}

We can assume that $n_0,n_1$ have zero means. If $\int n_0, \int n_1 \neq 0$, then we can consider the change of variable
\begin{align}\label{variable}
    u(x,t)\mapsto e^{i \Big(\frac{t^2}{4\pi} \int n_1  + \frac{t}{2\pi} \int n_0 \Big)} u(x,t);\: n(x,t) \mapsto n(x,t) - \frac{t}{2\pi} \int n_1 - \frac{1}{2\pi} \int n_0,
\end{align}
which can be directly checked to satisfy (\ref{eq:main}) with zero means in the new variable. By integrating the second equation of (\ref{eq:main}) over space, one obtains $\frac{d^2}{dt^2}\int_\mathbb{T} n=0$, and therefore the mean zero condition on $n_0,n_1$ allows us to make sense of $\lVert \partial_t n \rVert_{\dot{H}^{-1}}$ in the energy functional. We will use this idea extensively to obtain global solutions.

One expects a Hamiltonian system to have its Lagrangian counterpart via Legendre transform. Define
\begin{equation}\label{lagrangian}
    \mathscr{L} = \frac{i}{2}(\overline{u}\partial_t u - u \partial_t \overline{u}) - \alpha \partial_x u \partial_x \overline{u} - (\partial_x \nu) u \overline{u}+ \frac{1}{2\beta^2}(\partial_t \nu)^2 -\frac{1}{2}(\partial_x \nu)^2 - \epsilon^2 \partial_{xx}u \partial_{xx}\overline{u} - \frac{ \epsilon^2}{2}(\partial_{xx}\nu)^2,
\end{equation}
where $u$ is a complex field, $\overline{u}$, the conjugate field of $u$, and $\nu$, a real field where we impose $n \coloneqq \partial_x \nu$. The action functional $\mathcal{S}$ corresponding to $\mathscr{L}$ is defined in the usual way as follows:
\begin{equation}\label{action}
    S = \int \mathscr{L}(u,\overline{u},\nu,\partial_\mu u, \partial_\mu \overline{u}, \partial_\mu \nu)dxdt,
\end{equation}
where $\partial_\mu$ denotes higher derivatives. To look for the critical points of $\mathcal{S}$, we impose

\begin{equation}
    \delta \mathcal{S}=0,
\end{equation}
which amounts to solving the Euler-Lagrange equations corresponding to the given fields. One can check that the Euler-Lagrange equation for $u$ yields the first equation of (\ref{eq:main}), and the spatial derivative of the Euler-Lagrange equation for $\nu$ yields the second equation of (\ref{eq:main}).

Let $D$ be a Banach space and for $T\in (0,\infty)$, let $C([0,T],D)$ denote the Banach space of $D$-valued continuous (in time) functions with $\lVert u \rVert_{C_T D}\coloneqq\sup\limits_{t\in [0,T]} \lVert u(t)\rVert_D <\infty$. When $T=\infty$, consider $C_{loc}([0,\infty),D)$ where we only require continuity in $t$. We wish to obtain a strong solution $(u,n,\partial_t n)$ to (\ref{eq:main}) and by this we mean $(u,n,\partial_t n) \in C([0,T], H^{s,l})$ for some $T>0$ that satisfies the Duhamel's principle
\begin{align}\label{eq:4.9}
    u(t) &= U_\epsilon(t)u_0 - i \int_0^t U_\epsilon(t-t^\prime) (un)(t^\prime)dt^\prime\\
    n(t) &= \partial_t V_\epsilon(t)n_0 + V_\epsilon(t)n_1+ \beta^2 \int_0^t V_\epsilon(t-t^\prime) \partial_{xx}(|u|^2)(t^\prime)dt^\prime,\nonumber
\end{align}
where $U_\epsilon(t),V_\epsilon(t),\partial_t V_\epsilon(t)$ for $\epsilon\geq 0$ are defined via Fourier multipliers as
\begin{align}\label{linop}
\reallywidehat{U_\epsilon(t)u}(k) &= e^{-it (\alpha k^2 + \epsilon^2 k^4)} \widehat{u}(k),\nonumber\\
\reallywidehat{V_\epsilon(t)n_1}(k)&=\begin{cases}
    \frac{\sin (\beta |k|\langle \epsilon k \rangle t)}{\beta |k|\langle \epsilon k \rangle} \widehat{n}_1(k),& k \neq 0\\
    t\cdot\widehat{n}_1(k),              & k=0,
\end{cases}\nonumber\\
\reallywidehat{\partial_t V_\epsilon(t)n_0}(k) &= \cos (\beta |k|\langle \epsilon k \rangle t) \widehat{n}_0(k).
\end{align}

To obtain low-regularity well-posedness, we define the modified Bourgain norm adapted to the linear operators of interest. Take a complex-valued $f \in C^\infty_c (\mathbb{T}\times \mathbb{R})$ and define

\begin{equation}
    \lVert f \rVert_{X^{s,b}_S} = \lVert \langle k \rangle^s \langle \tau + \alpha k^2 + \epsilon^2 k^4\rangle^b \widehat{f}(k,\tau)\rVert_{L^2_\tau l^2_k};\:\lVert f \rVert_{X^{l,b}_W} = \lVert \langle k \rangle^l \langle |\tau| -\beta |k|\langle \epsilon k \rangle\rangle^b \widehat{f}(k,\tau)\rVert_{L^2_\tau l^2_k},
\end{equation}
from which we define $X^{s,b}_S$ and $X^{l,b}_W$ as the closure of $C^{\infty}_c(\mathbb{T}\times\mathbb{R})$ with respect to the norms introduced above, respectively. We refer to expressions such as $\langle \tau + \alpha k^2 + \epsilon^2 k^4\rangle$ and $\langle |\tau| -\beta |k|\langle \epsilon k \rangle\rangle$ as dispersive weights. As in literature, we refer to these spaces as Bourgain spaces. As usual, $\widehat{f}$ denotes the spacetime Fourier transform

\begin{equation}
    \widehat{f}(k,\tau) = \frac{1}{4\pi^2}\int f(x,t)e^{-i(kx + \tau t)}dxdt,
\end{equation}
and whenever it is clear, we use $\widehat{f}$ to denote either the spatial Fourier transform or the spacetime transform. Although for $b>\frac{1}{2}$, we have
\begin{equation}
X_S^{s,b}\hookrightarrow C(\mathbb{R},H^s),   
\end{equation}
we are interested in the endpoint case $b=\frac{1}{2}$ where the continuous embedding into $C(\mathbb{R},H^s)$ fails. Motivated by the Fourier inversion theorem, we augment the norm and consider
\begin{equation}
    \lVert f \rVert_{Y^s_S}= \lVert f \rVert_{X^{s,\frac{1}{2}}_S}+ \lVert \widehat{f}(k,\tau)\langle k \rangle^s \rVert_{l^2_k L^1_\tau};\:\lVert f \rVert_{Y^l_W} = \lVert f \rVert_{X^{l,\frac{1}{2}}_W}+ \lVert \widehat{f}(k,\tau)\langle k \rangle^s \rVert_{l^2_k L^1_\tau},
\end{equation}
from which we can recover the desired continuous embedding, that is,
\begin{equation}
    Y_S^s \hookrightarrow C(\mathbb{R},H^s),
\end{equation}
and similarly for $Y_W^l$. To control the Duhamel term coming from the nonlinearities, we consider the companion spaces to $Y_S^s, Y_W^l$:
\begin{align}
    \lVert f \rVert_{Z^s_S} = \lVert f \rVert_{X^{s,-\frac{1}{2}}_S} + \left|\left| \frac{\langle k \rangle^s}{\langle \tau+\alpha k^2 + \epsilon^2 k^4\rangle}  \widehat{f}(k,\tau) \right|\right|_{l^2_k L^1_\tau};\:\lVert f \rVert_{Z^l_W} = \lVert f \rVert_{X^{l,-\frac{1}{2}}_W} + \left|\left| \frac{\langle k \rangle^l}{\langle |\tau| -\beta |k|\langle \epsilon k \rangle\rangle}  \widehat{f}(k,\tau) \right|\right|_{l^2_k L^1_\tau}.
\end{align}

To obtain solutions for small time, we further define the time-restricted space for $T>0$
\begin{equation}
    \lVert f \rVert_{X_{S,T}^{s,b}} = \inf_{\Tilde{f}=f,\:t\in [0,T]}\lVert \Tilde{f}\rVert_{X_{S}^{s,b}},
\end{equation}
where such restriction for other Bourgain spaces can be defined analogously. 

We say $A \lesssim B$ or $A \gtrsim B$ if there exists some $C>0$ such that $A \leq C B$ or $A \geq CB$, and $A \simeq B$ if $A \lesssim B$ and $A \gtrsim B$. Given $A_{\pm}$, we denote $\sum_{\pm}A_{\pm}\coloneqq A_+ + A_-$. We let $\psi \in C_c^\infty (\mathbb{R})$ to be a smooth cutoff with a compact support in $[-2,2]$ and $\psi(t)=1$ for all $t\in [-1,1]$. For $b\in\mathbb{R}$, we write $b\pm$ to denote $b\pm\epsilon^\prime$ for some universal $\epsilon^\prime\ll 1$. Though not necessary, we assume the perturbation parameter $\epsilon\leq 1$ to make the exposition clearer.
\section{Linear estimates.}\label{linearestimates}
Here we assume $\alpha,\beta,\epsilon>0$.

\begin{lemma}[Homogeneous Estimates]\label{linest1}
For $s,l \in \mathbb{R}$,
\begin{align*}
\lVert U_\epsilon (t) u_0 \rVert_{H^s}=\lVert u_0 \rVert_{H^s}&;\:\lVert \psi(t) U_\epsilon(t) u_0\rVert_{Y^s_S} \leq c_1(\psi ) \lVert u_0 \rVert_{H^s},\forall \epsilon \geq 0.\\ 
\lVert \partial_t V_\epsilon (t) n_0 \rVert_{H^l}\leq\lVert n_0 \rVert_{H^l} &;\:\lVert \psi(t) \partial_t V_\epsilon(t) n_0\rVert_{Y^l_W}\leq c_2(\psi ) \lVert n_0 \rVert_{H^l},\forall \epsilon \geq 0.\\
\lVert V_\epsilon (t) n_1 \rVert_{H^l}\leq c\Big(t + \frac{1}{\beta \epsilon}\Big)\lVert n_1 \rVert_{H^{l-2}}&;\:\lVert \psi(t) V_\epsilon(t)n_1\rVert_{Y^l_W}\leq c_3(\psi)\Big(1+\frac{1}{\beta \epsilon}\Big)\lVert n_1 \rVert_{H^{l-2}},\\
\lVert V_0 (t) n_1 \rVert_{H^l}\leq c\Big(t + \frac{1}{\beta }\Big)\lVert n_1 \rVert_{H^{l-1}}&;\: \lVert \psi(t) V_0(t)n_1\rVert_{Y^l_W}\leq c_4(\psi)\Big(1+\frac{1}{\beta}\Big)\lVert n_1 \rVert_{H^{l-1}}.
\end{align*}
\end{lemma}
\begin{proof}
The first line of inequalities follows from the unitarity of Schr\"odinger operator; see \cite[Lemma 2.8]{tao2006nonlinear}. A similar argument can be used to show the other inequalities.
\end{proof}

\begin{lemma}[Duhamel Estimates]
For $s,l\in\mathbb{R}$ and $\rho \in [0,1]$,
\begin{align*}
     \left|\left| \psi(t)\int_0^t U_\epsilon(t-t^\prime) F(t^\prime) dt^\prime \right|\right|_{Y^s_S} &\leq c_1(\psi) \lVert F \rVert_{Z^s_S},\\
     \left|\left| \psi(t)\int_0^t V_\epsilon(t-t^\prime) |\nabla|^{2-\rho}F(t^\prime) dt^\prime \right|\right|_{Y^l_W} &\leq c_2(\psi)c_2(\rho,\beta,\epsilon) \lVert F \rVert_{Z^l_W},\\
    \left|\left| \psi(t)\int_0^t \partial_t V_\epsilon(t-t^\prime) D^{2-\rho}F(t^\prime) dt^\prime \right|\right|_{Y^{l-2}_W} &\leq c_3(\psi) \lVert F \rVert_{Z^l_W}.
\end{align*}
\end{lemma}
\begin{proof}
The first inequality is standard in literature; see \cite[Proposition 2.12]{tao2006nonlinear}. The second and third are proved similarly where 
    \begin{center}
        $c_2(\rho,\beta,\epsilon) = \begin{cases}
    \frac{\Big(\frac{1-\rho}{\rho \epsilon^2}\Big)^{\frac{1-\rho}{2}}}{\beta \rho^{-1/2}},& \rho \in (0,1)\\
    \frac{1}{\beta \epsilon},              & \rho=0\\
    \frac{1}{\beta}, & \rho=1.
\end{cases}$
    \end{center}
\end{proof}

The following estimates allow us to extract a (small) positive time factor, which is applied to obtain local well-posedness.

\begin{lemma}\label{l:4.3}
Let $T\leq 1$, $s,l \in \mathbb{R}$, and $-\frac{1}{2}<b \leq b^\prime < \frac{1}{2}$. Then
\begin{align*}
    \lVert \psi(t/T)u\rVert_{X^{s,b}_S}&\lesssim_{\psi,b,b^\prime} T^{b^\prime-b} \lVert u \rVert_{X^{s,b^\prime}_S},\\
    \lVert \psi(t/T)u\rVert_{X^{l,b}_W}&\lesssim_{\psi,b,b^\prime} T^{b^\prime-b} \lVert u \rVert_{X^{l,b^\prime}_W}.
\end{align*}
\end{lemma}
\begin{proof}
The first inequality follows from \cite[Lemma 2.11]{tao2006nonlinear}. A similar argument can be used to show the second inequality.
\end{proof}

\section{Nonlinear estimates.}\label{nonlinearestimates}

Here we fix $\alpha,\beta,\epsilon>0$ and $T \in (0,1]$.
\begin{proposition}\label{p:51}
For $0<\rho\leq 1$, suppose $s\geq 0,\:-1\leq l \leq 2s+1-\rho,\:-2+\rho \leq s-l\leq 2$ and $b\in (\frac{1}{6},\frac{1}{2}]$. Then there exists $C=C(\alpha,\beta,\epsilon,\rho,s,l,b)>0$ such that
\begin{align*}
    \lVert un \rVert_{X^{s,-\frac{1}{2}}_S} &\leq C (\lVert u \rVert_{X^{s,b}_S}\lVert n \rVert_{X^{l,\frac{1}{2}}_W}+\lVert u \rVert_{X^{s,\frac{1}{2}}_S}\lVert n \rVert_{X^{l,b}_W}),\\
    \lVert D^\rho(u \overline{v})\rVert_{X^{l,-\frac{1}{2}}_W}&\leq C( \lVert u \rVert_{X^{s,b}_S}\lVert v \rVert_{X^{s,\frac{1}{2}}_S}+\lVert u \rVert_{X^{s,\frac{1}{2}}_S}\lVert v \rVert_{X^{s,b}_S}).
\end{align*}
\end{proposition}

\begin{proposition}\label{p:5.2}
Assume the hypotheses of proposition \ref{p:51}. Then,
\begin{align*}
    \left|\left| \frac{\langle k \rangle^s}{\langle \tau+\alpha k^2 + \epsilon^2 k^4\rangle}  \reallywidehat{un}(k,\tau) \right|\right|_{l^2_k L^1_\tau} &\lesssim  \lVert u \rVert_{X^{s,b}_S}\lVert n \rVert_{X^{l,\frac{1}{2}}_W}+\lVert u \rVert_{X^{s,\frac{1}{2}}_S}\lVert n \rVert_{X^{l,b}_W},\\
    \left|\left| \frac{\langle k \rangle^l}{\langle |\tau| - \beta |k|\langle \epsilon k \rangle\rangle} \reallywidehat{D^\rho (u\overline{v})}(k,\tau)\right|\right|_{l^2_k L^1_\tau}&\lesssim \lVert u \rVert_{X^{s,b}_S}\lVert v \rVert_{X^{s,\frac{1}{2}}_S}+\lVert u \rVert_{X^{s,\frac{1}{2}}_S}\lVert v \rVert_{X^{s,b}_S}.
\end{align*}
\end{proposition}

\begin{corollary}\label{c:5.1}
Assume the same hypotheses as before. Then for some $\theta \in (0,\frac{1}{3})$,
\begin{align*}
    \lVert un \rVert_{Z^s_S} &\lesssim T^\theta \lVert u \rVert_{Y^s_S}\lVert n \rVert_{Y^l_W},\\
    \lVert D^\rho (|u|^2)\rVert_{Z^l_W}&\lesssim T^\theta \lVert u \rVert_{Y^s_S}^2.
\end{align*}
\end{corollary}
\begin{remark}
Note that corollary \ref{c:5.1} is an immediate consequence of proposition \ref{p:51}, proposition \ref{p:5.2}, and lemma \ref{l:4.3}. Here we will not be concerned with obtaining an optimal range for $b$ or $\theta$.
\end{remark}

\begin{remark}
Note that the LHS in proposition \ref{p:5.2} is controlled by the same term as in proposition \ref{p:51}. Though proposition \ref{p:5.2} is proved in a similar way to proposition \ref{p:51} via the duality trick, one needs to be wary of the $L^1_\tau$-estimate; see \cite{takaoka1999well} for more detail.
\end{remark}

\begin{remark}
The method of direct estimation by the Cauchy-Schwarz inequality does not seem to work, at least directly, when $\rho=0$. One can check that the $\tau_1$-integral in (\ref{eq:5.27}) is not justified. In fact if $k=0$, then $IV=\infty$ by a direct computation.
\end{remark}

As a corollary, we show

\begin{theorem}\label{th:main1}
If $(s,l)\in \Omega_L$,
then (\ref{eq:main}) is locally well-posed; that is, there exists $T = T(\lVert u_0\rVert_{H^s}, \lVert n_0\rVert_{H^l},\lVert n_1\rVert_{H^{l-2}})>0$ and a unique $(u,n,\partial_t n) \in Y^s_{S,T}\times Y^l_{W,T}\times Y^{l-2}_{W,T}$ that satisfies (\ref{eq:main}). Further, if $T^\prime \in (0,T)$, then there exists a neighborhood $B\subseteq H^{s,l}$ around $(u_0,n_0,n_1)$ such that the data-to-solution map $(u_0,n_0,n_1)\mapsto (u,n,\partial_t n)$ is Lipschitz-continuous from $B$ to $Y^s_{S,T}\times Y^l_{W,T}\times Y^{l-2}_{W,T}$.
\end{theorem}

\begin{proof}[proof of theorem \ref{th:main1}]

Define $X=Y^s_{S,T} \times Y^l_{W,T} \times Y^{l-2}_{W,T}$ with $\lVert (u,n,\partial_t n)\rVert \coloneqq \lVert u \rVert_{X^s_{S,T}}+ \lVert n \rVert_{Y^l_{W,T}}+ \lVert \partial_t n\rVert_{Y^{l-2}_{W,T}}$, and $X(R)= \left\{(u,n,\partial_t n)\in X: \lVert (u,n,\partial_t n) \rVert \leq R\sigma \right\}$ for $R>0$ to be determined and $\sigma = \lVert u_0 \rVert_{H^s}+\lVert n_0 \rVert_{H^l}+\lVert n_1 \rVert_{H^{l-2}}$. Further, define a map $\Gamma (u,n,\partial_t n) = (\Gamma_1 (u,n), \Gamma_2 (u), \Gamma_3 (u))$ on $X(R)$ where
\begin{align}
     \Gamma_1(u,n)(t) &= \psi(t)U(t)u_0 - i \psi(t)\int_0^t U(t-t^\prime) (un)(t^\prime)dt^\prime,\nonumber\\
    \Gamma_2 (u)(t) &= \psi(t)\partial_t V(t)n_0+ \psi(t)V(t)n_1 + \beta^2 \psi(t)\int_0^t V(t-t^\prime) \partial_{xx}(|u|^2)(t^\prime)dt^\prime,
\end{align}
and $\Gamma_3 (u)(t)= \partial_t \Gamma_2(u)(t)$ is defined in the sense of distribution. If $(s,l)\in \Omega_L$, pick $\rho>0$ sufficiently small such that the hypotheses of proposition \ref{p:51} are fulfilled. Then by corollary \ref{c:5.1}, there exists $\theta>0$ such that

\begin{align}
    \lVert \Gamma(u,n,\partial_t n)\rVert &\lesssim \sigma + T^\theta \lVert u \rVert_{Y^s_{S,T}}\lVert n\rVert_{Y^l_{W,T}} + T^\theta \lVert u \rVert_{Y^s_{S,T}}^2 \leq \sigma + 2T^\theta R^2 \sigma^2.
\end{align}

If $R$ is chosen sufficiently big (depending on the choice of given parameters), then by choosing $0<T \lesssim \sigma^{-\theta}$, we conclude that $\Gamma$ maps into $X(R)$. Similarly given $(u_1,n_1,\partial_t n_1), (u_2,n_2,\partial_t n_2) \in X$, we have
\begin{align}
    \lVert \Gamma(u_1,n_1,\partial_t n_1) - \Gamma(u_2,n_2,\partial_t n_2) \rVert &\lesssim T^\theta \Big(\lVert u_1 \rVert_{Y^s_{S,T}}\lVert n_1-n_2\rVert_{Y^l_{W,T}} + \lVert n_2 \rVert_{Y^l_{W,T}}\lVert u_1-u_2\rVert_{Y^s_{S,T}} \Big)\nonumber\\
    &+T^\theta\Big(\lVert u_1\rVert_{Y^s_{S,T}}\lVert u_1-u_2\rVert_{Y^s_{S,T}}+\lVert u_2\rVert_{Y^s_{S,T}}\lVert u_1-u_2\rVert_{Y^s_{S,T}}\Big)\nonumber\\
    &\lesssim T^\theta R \sigma \lVert (u_1,n_1,\partial_t n_1)-(u_2,n_2,\partial_t n_2)\rVert,
\end{align}
and by choosing $T \lesssim \sigma^{-\theta}$ sufficiently small, $\Gamma$ defines a contraction operator on $X(R)$, and hence there exists a unique $(u,n,\partial_t n) \in X \hookrightarrow C([0,T], H^{s,l})$ that satisfies (\ref{eq:main}). Local lipschitz-continuity of the data-to-solution map follows from the contraction mapping principle.
\end{proof}

\begin{remark}
By construction, $\partial_t n$ is indeed the distributional derivative of $n$. By defining $X$ to include $\partial_t n$, not just $(u,n)$, we obtain the continuity of $\partial_t n$ in $t$ as a result of the contraction argument.
\end{remark}

Now the goal is to show the boundedness of the multilinear operators corresponding to the nonlinear terms which, at a technical level, involves directly estimating a $L^\infty L^1$-norm of a function defined on the spacetime Fourier space in different regions depending on which dispersive weight is most dominant. Observing that
\begin{equation}
    (\tau + \alpha k^2 +\epsilon^2 k^4)-(\tau_1 + \alpha k_1^2 + \epsilon^2 k_1^4)-(\tau_2 \pm \beta k_2 \langle \epsilon k_2\rangle) = (k-k_1)\Big((k+k_1)(\alpha + \epsilon^2 (k^2+k_1^2))\mp \beta \langle \epsilon (k-k_1)\rangle\Big),
\end{equation}
we obtain
\begin{equation}\label{eq:55}
    \max\Big(|\tau + \alpha k^2 +\epsilon^2 k^4|,|\tau_1 + \alpha k_1^2 + \epsilon^2 k_1^4|,\left||\tau_2|-  \beta |k_2| \langle \epsilon k_2 \rangle\right|\Big)\geq \frac{1}{3}|k-k_1|\left|(k+k_1)(\alpha + \epsilon^2 (k^2+k_1^2))\mp \beta \langle \epsilon (k-k_1)\rangle\right|,
\end{equation}
where the sign on the RHS of (\ref{eq:55}) depends on $\tau_2,k_2$. Since this subtlety does not affect our subsequent analysis, we do not keep track of the sign. For notational convenience, we define $h(k,k_1) = (k+k_1)(\alpha + \epsilon^2 (k^2+k_1^2))\mp \beta \langle \epsilon (k-k_1)\rangle$.

\begin{lemma}\parencite[Lemma 3.3]{erdougan2013smoothing}\label{l:51}
If $\delta \geq \gamma \geq 0$ and $\delta + \gamma>1$, then
\begin{equation*}
\int \dfrac{d\tau}{\langle \tau-a_1 \rangle^\delta \langle \tau-a_2 \rangle^\gamma} \lesssim \langle a_1-a_2 \rangle^{-\gamma} \phi_\delta (a_1-a_2),\:\text{where}    
\end{equation*}
    \begin{center}
        $\phi_\delta (a) \simeq \begin{cases}
    1,& \delta > 1\\
    \log (1+\langle a \rangle),              & \delta=1\\
    \langle a \rangle^{1-\delta}, & \delta<1.
\end{cases}$
    \end{center}
\end{lemma}

\begin{lemma}\label{l:52}
For all $e_1>\frac{1}{4}, e_2>\frac{1}{3}$,
\begin{align*}
    \sigma_1(k,\tau)&\coloneqq\sum_{k_1 \neq k; \pm}\frac{1}{\langle \epsilon^2 k_1^4+\alpha k_1^2 + \tau \pm \beta (k-k_1)\langle \epsilon (k-k_1)\rangle\rangle^{e_1}}\leq c_1(\alpha,\beta,\epsilon,e_1)<\infty,\\
    \sigma_2(k,\tau)&\coloneqq \sum_{k_1} \frac{1}{\langle k_1^3 - \frac{3k}{2}k_1^2 + \Big(\frac{\alpha + 2\epsilon^2 k^2}{2\epsilon^2}\Big)k_1 + \frac{\tau - \alpha k^2-\epsilon^2 k^4}{4\epsilon^2 k} \rangle^{e_2}}\leq c_2(\alpha,\beta,\epsilon,e_2)<\infty,\: \text{if}\: k \neq 0.
\end{align*}
\end{lemma}

\begin{lemma}\label{l:53}
There exist $C(\alpha,\beta,\epsilon),c(\alpha,\beta,\epsilon)>0$ such that for all $(k,k_1)\in \mathbb{Z}^2$ that satisfies $\left\{|k|\geq C(\alpha,\beta,\epsilon)\right\}\cup \left\{|k_1|\geq C(\alpha,\beta,\epsilon)\right\}$, we have
\begin{equation*}
    |h(k,k_1)| \geq c(\alpha,\beta,\epsilon) |k-k_1|.
\end{equation*}
\end{lemma}

\begin{lemma}\label{l:54}
There exists $C(\alpha,\beta,\epsilon)>0$ such that if $\left\{0\neq |k|\geq 2 |k_1|\right\}\cap \left\{|k|\geq C(\alpha,\beta,\epsilon)\right\}$, we have $|k-k_1||h(k,k_1)| \gtrsim |k|^4$. Similarly if $\left\{0\neq \frac{|k_1|}{2}\geq |k|\right\}\cap \left\{|k_1|\geq C(\alpha,\beta,\epsilon)\right\}$, then $|k-k_1||h(k,k_1)| \gtrsim |k_1|^4$.
\end{lemma}

\begin{proof}[proof of lemma \ref{l:52}]
The second inequality can be proven in a similar way as to \cite[lemma 3(c)]{erdougan2013smoothing}. For the first inequality, there exists $c>0$ independent of $k,k_1$ such that
\begin{align*}
\left|(k-k_1)\langle \epsilon (k-k_1)\rangle - (k-k_1)|\epsilon (k-k_1)|\right|\leq c.
\end{align*}
Hence the term $\langle \epsilon (k-k_1)\rangle$ in the summation can be replaced with $|\epsilon (k-k_1)|$. Then
\begin{align*}
    \sigma_1(k,\tau) \leq \sum_{k_1 \neq k; \pm}\frac{1}{\langle \epsilon^2 k_1^4+\alpha k_1^2 + \tau \pm \beta\epsilon (k-k_1)^2\rangle^{e_1}}\leq c^\prime,
\end{align*}
where the constant is independent of $k,\tau$ by an argument similar to \cite[lemma 3(c)]{erdougan2013smoothing}.
\end{proof}

\begin{proof}[proof of lemma \ref{l:53}]

Assume $k\neq k_1$. For a fixed $k\in \mathbb{Z}$, let $r_{\mp}(k) \in \mathbb{R}$ be the unique real-root of $h(k,\cdot)$ where $r_-(k)$ corresponds to the minus sign in $h(k,\cdot)$, and similarly for $r_+(k)$; we drop the $\mp$-subscript. Noting that $h$ is symmetric in both arguments, it suffices to assume $|k|\geq C(\alpha,\beta,\epsilon)$ where
\begin{align}\label{lemlowerbound}
C(\alpha,\beta,\epsilon)\coloneqq \max\Big(C_1(\alpha,\beta,\epsilon), \sqrt{\frac{3\sqrt{2}\beta}{\epsilon}},\frac{1}{3\epsilon}\Big)    
\end{align}
where for all $|k|\geq C_1(\alpha,\beta,\epsilon)>0$, we have $\beta \langle \epsilon k \rangle<|k|(\alpha + \epsilon^2 k^2)$.

We first show that for $k$ sufficiently big, $r(k) \notin \mathbb{Z}$. For $k \in \mathbb{Z}$, consider the graphs of $k_1 \mapsto (k+k_1)(\alpha + \epsilon^2 (k^2+k_1^2))$ and $k_1\mapsto \pm \beta \langle \epsilon (k-k_1)\rangle$. If we require that the $y$-intercept of the cubic polynomial is greater (in magnitude) than that of the square-root term, i.e., $\beta \langle \epsilon k \rangle < |k|(\alpha+\epsilon^2k^2)$, then $r(k) \in [-c_2k,0]$ for $k>0$ and $r(k) \in [0,-c_2k]$ for $k<0$ where $c_2 = c_2(\alpha,\beta,\epsilon)>0$.

Now we claim $\lim\limits_{|k|\rightarrow \infty }|r(k)+k|=0$. From $h(k,r(k))=0$, we obtain
\begin{align}
    |r(k)+k| = \left|\frac{\beta \langle \epsilon (k-r(k)\rangle)}{\alpha + \epsilon^2 (k^2+r(k)^2)}\right| \lesssim \frac{\beta \epsilon |k-r(k)|}{\alpha+\epsilon^2k^2}\leq \frac{(1+c_2)\beta \epsilon |k|}{\alpha +\epsilon^2 k^2}\xrightarrow[|k|\rightarrow\infty]{}0.
\end{align}

Hence if $|k|$ is sufficiently big and $r(k)\in \mathbb{Z}$, then $r(k)=-k$, which cannot be since $|h(k,-k)| = \beta \langle 2 \epsilon k \rangle\geq \beta$. For $k\in \mathbb{Z}$, to show $\inf\limits_{k_1\in \mathbb{Z}}|h(k,k_1)|$ is attained at $k_1 = -k$, note that from standard calculus,
\begin{equation}
    \partial_{k_1}h(k,k_1) = 3\epsilon^2 k_1^2 + 2 \epsilon^2 k k_1 + \alpha +\epsilon^2k^2 \pm \frac{\beta \epsilon^2 (k-k_1)}{\langle \epsilon (k-k_1)\rangle}\geq \alpha+\frac{2}{3}\epsilon^2k^2\pm \frac{\beta \epsilon^2 (k-k_1)}{\langle \epsilon (k-k_1)\rangle},
\end{equation}
and since $\frac{\beta \epsilon^2 |k-k_1|}{\langle \epsilon (k-k_1)\rangle}\leq \beta\epsilon$, we have $\partial_{k_1}h \geq \alpha$ by (\ref{lemlowerbound}), and hence 
\begin{align*}
\inf\limits_{|k|\geq C(\alpha,\beta,\epsilon),\: (k,k_1)\in \mathbb{Z}^2}|h(k,k_1)|\geq \beta.    
\end{align*}

If $|k-k_1| \leq 3 |k|$, then $\inf\limits_{|k|\geq C(\alpha,\beta,\epsilon),\: (k,k_1)\in \mathbb{Z}^2}|\frac{h(k,k_1)}{k-k_1}|\geq \frac{\beta}{3 C(\alpha,\beta,\epsilon)}$. If $|k-k_1|\geq 3 |k|$, then $|k_1| \geq 2 |k|$, $|k+k_1| \geq \frac{|k_1|}{2}$, and $|k-k_1| \leq \frac{3|k_1|}{2}$. Furthermore
\begin{align}
    \left|\frac{h(k,k_1)}{k-k_1}\right| \geq \frac{1}{3}(\alpha+\epsilon^2(k^2+k_1^2)) - \left|\beta \frac{\langle \epsilon (k-k_1)\rangle}{k-k_1}\right|\geq \frac{1}{3}(\alpha+\epsilon^2(k^2+k_1^2)) - \sqrt{2}\beta \epsilon \geq \frac{\alpha}{3},
\end{align}
where the last inequality is by (\ref{lemlowerbound}).
\end{proof}

\begin{proof}[proof of lemma \ref{l:54}]
Since $h$ is symmetric in $k,k_1$, it suffices to prove the first statement. For $|k|\geq 2|k_1|$, we have $|k\pm k_1|\geq \frac{|k|}{2}$. If we further assume $|k|\geq \frac{2}{\epsilon}$, we have $\epsilon |k-k_1|\geq \frac{\epsilon}{2}|k|\geq 1$, and therefore
\begin{equation}
    \beta \langle \epsilon (k-k_1)\rangle\leq \sqrt{2}\beta \epsilon |k-k_1|\leq \sqrt{2}\beta\epsilon (|k|+|k_1|)\leq \frac{3}{\sqrt{2}}\beta\epsilon |k|.
\end{equation}
By the triangle inequality, if $|k|\geq 10\max(\frac{1}{\epsilon},\sqrt{\frac{\beta}{\epsilon}})$,
\begin{align}
    |k-k_1||h(k,k_1)|\geq \frac{|k|}{2}\Big(\frac{\epsilon^2}{2}|k|^3 - \frac{3}{\sqrt{2}}\beta \epsilon |k|\Big) \geq \frac{\epsilon^2}{8}|k|^4. 
\end{align}
\end{proof}

\begin{proof}[proof of proposition \ref{p:51}]
Though the main idea of this proof follows closely that of \cite{takaoka1999well}, we include a full proof here to address any subtleties that rise from the fourth-order perturbation. To use the duality argument, let $w \in L^2_{k,\tau}$, $\lVert w \rVert_{L^2}=1$ and $w\geq 0$. Since
\begin{equation}
    \lVert un \rVert_{X^{s,-\frac{1}{2}}_S} = \left|\left| \frac{\langle k \rangle^s}{\langle \tau+\alpha k^2 + \epsilon^2 k^4\rangle^{1/2}}\sum_{k_1+k_2=k}\int_{\tau_1+\tau_2=\tau}\hat{u}(\tau_1,k_1)\hat{n}(\tau_2,k_2)\right|\right|_{l^2_k L^2_\tau},
\end{equation}
it suffices to estimate
\begin{equation}\label{eq:52}
    E\coloneqq\sum_{k_1+k_2-k=0}\int_{\tau_1+\tau_2-\tau=0} \frac{\langle k \rangle^s \langle k_1 \rangle^{-s}\langle k_2 \rangle^{-l}f(\tau_1,k_1)g(\tau_2,k_2)w(\tau,k)}{\langle \tau+\alpha k^2 + \epsilon^2 k^4\rangle^{1/2}\langle \tau_1+\alpha k_1^2 + \epsilon^2 k_1^4\rangle^{b_1}\langle |\tau_2| - \beta |k_2|\langle \epsilon k_2 \rangle\rangle^{b_2}},
\end{equation}
where
\begin{equation}
    f(\tau,k) = |\hat{u}(\tau,k)|\langle k \rangle^s \langle \tau + \alpha k^2 +\epsilon^2 k^4\rangle^{b_1};\: g(\tau,k) = |\hat{n}(\tau,k)|\langle k \rangle^l\langle |\tau|-\beta |k|\langle \epsilon k \rangle\rangle^{b_2},
\end{equation}
and $b_1,b_2 \leq \frac{1}{2}$. By direct computation, one can rule out $k_2=0$, and hence we assume the sum is over $k_2 \neq 0$, or equivalently, $k_1 \neq k$. 

\textbf{I.} $\max\Big(|\tau + \alpha k^2 +\epsilon^2 k^4|,|\tau_1 + \alpha k_1^2 + \epsilon^2 k_1^4|,\left||\tau_2| - \beta |k_2| \langle \epsilon k_2 \rangle\right|\Big)=|\tau + \alpha k^2 +\epsilon^2 k^4|$.

Let $b_1=b_2 =b= \frac{1}{2}-$. Applying the Cauchy-Schwarz inequality in variables $(k_1,\tau_1)$ and $(k_2,\tau_2)$, followed by the Young's inequality, it suffices to show

\begin{equation}
    \sup_{\tau,k} I \coloneqq \sup_{\tau,k} \frac{\langle k \rangle^{2s}}{\langle \tau + \alpha k^2 + \epsilon^2 k^4\rangle} \sum_{k_1\neq k} \int \frac{d\tau_1}{\langle k_1 \rangle^{2s}\langle k-k_1\rangle^{2l}\langle \tau_1 + \alpha k_1^2 + \epsilon^2k_1^4\rangle^{2b}\langle |\tau-\tau_1|-\beta |k-k_1|\langle \epsilon(k-k_1)\rangle\rangle^{2b}}<\infty,
\end{equation}
since
\begin{equation}
    E \lesssim \Big(\sup_{\tau,k}I\Big)^{1/2}\lVert u \rVert_{X^{s,b}_S}\lVert n \rVert_{X^{l,b}_W}\lVert w \rVert_{L^2_{k,\tau}}.
\end{equation}

Let $\frac{|k|}{2}\leq |k_1|\leq 2 |k|$. Integrating in $\tau_1$ via lemma \ref{l:51} and noting that $\langle k \rangle^{2s}\langle k_1 \rangle^{-2s}\simeq 1$, we have

\begin{equation}\label{eq:5.8}
    I \lesssim \sum_{k_1 \neq k; \pm} \frac{1}{\langle \tau + \alpha k^2 + \epsilon^2 k^4\rangle \langle k-k_1 \rangle^{2l}\langle \epsilon^2 k_1^4 + \alpha k_1^2+\tau \pm \beta (k-k_1)\langle \epsilon(k-k_1)\rangle\rangle^{4b-1}}.
\end{equation}

If $|k|\lesssim 1$, then $\langle k-k_1\rangle \simeq 1$, and therefore the sum above is finite by lemma \ref{l:52}. On the other hand, by lemma \ref{l:53} if $|k|\geq C(\alpha,\beta,\epsilon)$,
\begin{equation}
    I \lesssim \sum_{k_1\neq k; \pm} \frac{1}{\langle k-k_1\rangle^{2l+2}\langle  \epsilon^2 k_1^4 + \alpha k_1^2+ \tau \pm \beta (k-k_1)\langle \epsilon(k-k_1)\rangle\rangle^{4b-1}}\leq c_1.
\end{equation}
since $l\geq -1$. Now let $|k|\geq 2|k_1|$. In this region, we have $\frac{|k|}{2}\leq |k-k_1|\leq \frac{3|k|}{2}$ and by lemma \ref{l:54}, if $|k|\geq C(\alpha,\beta,\epsilon)$, then $\langle \tau + \alpha k^2 + \epsilon^2 k^4\rangle \gtrsim \langle k \rangle^4$ and
\begin{equation}
    I \lesssim \sum_{k_1\neq k; \pm} \frac{\langle k_1\rangle^{-2s}}{\langle k \rangle^{2l-2s+4}\langle \epsilon^2 k_1^4 + \alpha k_1^2 + \tau \pm \beta (k-k_1)\langle \epsilon(k-k_1)\rangle\rangle^{4b-1}}\leq c_1,
\end{equation}
since $s-l\leq 2$ and $s\geq 0$. If $|k|\lesssim 1$, then by lemma \ref{l:52}, $I \lesssim \sigma_1 \leq c_1$. Lastly if $\frac{|k_1|}{2}\geq |k|$, then $\frac{|k_1|}{2}\leq |k-k_1|\leq \frac{3|k_1|}{2}$ and by treating $|k_1|\geq C(\alpha,\beta,\epsilon)$ and $|k_1|\leq C(\alpha,\beta,\epsilon)$ separately as above, we have the desired uniform bound.

\textbf{II.} $\max\Big(|\tau + \alpha k^2 +\epsilon^2 k^4|,|\tau_1 + \alpha k_1^2 + \epsilon^2 k_1^4|,\left||\tau_2| - \beta |k_2| \langle \epsilon k_2 \rangle\right|\Big)=|\tau_1 + \alpha k_1^2 + \epsilon^2 k_1^4|$.

Arguing as above, it suffices to show
\begin{align}
    \sup_{\tau_1,k_1} II &\coloneqq \sup_{\tau_1,k_1}\frac{\langle k_1\rangle^{-2s}}{\langle \tau_1 + \alpha k_1^2 + \epsilon^2 k_1^4\rangle}\sum_{k \neq k_1}\int \frac{\langle k \rangle^{2s}d\tau}{\langle k-k_1\rangle^{2l}\langle \tau + \alpha k^2 + \epsilon^2 k^4\rangle \langle |\tau-\tau_1|-\beta |k-k_1|\langle \epsilon (k-k_1)\rangle\rangle^{2b_2}}\\
    &\lesssim \sup_{\tau_1,k_1} \frac{\langle k_1\rangle^{-2s}}{\langle \tau_1 + \alpha k_1^2 + \epsilon^2 k_1^4\rangle} \sum_{k\neq k_1; \pm}\frac{\langle k \rangle^{2s}}{\langle k-k_1\rangle^{2l}\langle \epsilon^2 k^4 + \alpha k^2 + \tau_1 \pm \beta (k_1-k)\langle \epsilon (k-k_1)\rangle\rangle^{2b_2-}},
\end{align}
where we set $b_1=\frac{1}{2}$. By lemma \ref{l:54}, if $|k|\geq 2|k_1|$ and $|k|\geq C(\alpha,\beta,\epsilon)$, then $\langle \tau_1 +\alpha k_1^2 + \epsilon^2 k_1^4\rangle \gtrsim \langle k \rangle^4$, and we have
\begin{equation}
    II \lesssim \sum_{k \neq k_1; \pm}\frac{\max(1,\langle k\rangle^{-2s})}{\langle k \rangle^{2l-2s+4}\langle \epsilon^2 k^4 + \alpha k^2 + \tau_1 \pm \beta (k_1-k)\langle \epsilon (k-k_1)\rangle\rangle^{2b_2-}}\leq c_1,
\end{equation}
and similarly, the desired uniform bound of $II$ follows if $|k|\geq 2|k_1|$ and $|k|\lesssim 1$; by applying lemma \ref{l:54} again, we can show that $II$ is uniformly bounded for $\frac{|k_1|}{2}\geq |k|$ by treating $|k|\leq C(\alpha,\beta,\epsilon)$ and $|k|\geq C(\alpha,\beta,\epsilon)$ separately. For $\frac{|k|}{2}\leq |k_1|\leq 2|k|$, we can argue as (\ref{eq:5.8}).

\textbf{III.} $\max\Big(|\tau + \alpha k^2 +\epsilon^2 k^4|,|\tau_1 + \alpha k_1^2 + \epsilon^2 k_1^4|,\left||\tau_2| - \beta |k_2| \langle \epsilon k_2\rangle\right|\Big)=\left||\tau_2| - \beta |k_2| \langle \epsilon k_2\rangle\right|$.\\

From (\ref{eq:55}), we have
\begin{equation}\label{eq:5.17}
    \left||\tau_2| - \beta |k_2| \langle \epsilon k_2\rangle\right| \gtrsim |k_2|\left|(2k-k_2)(\alpha +\epsilon^2(k^2+(k-k_2)^2))\mp \beta \langle \epsilon k_2\rangle\right|.
\end{equation}
If suffices to show $\sup\limits_{\tau_2,k_2} III <\infty$ where $b_2=\frac{1}{2}$ and
\begin{align}
    III &\coloneqq \frac{1}{\langle k_2 \rangle^{2l}\langle |\tau_2|-\beta |k_2|\langle \epsilon k_2\rangle\rangle}\sum_k \int \frac{\langle k \rangle^{2s} d\tau}{\langle k-k_2 \rangle^{2s}\langle \tau+\alpha k^2 +\epsilon^2 k^4\rangle \langle \tau-\tau_2+\alpha (k-k_2)^2 +\epsilon^2 (k-k_2)^4\rangle^{2b_1}}\\
    &\lesssim \sum_k \frac{\langle k \rangle^{2s}}{\langle k_2 \rangle^{2l}\langle k-k_2\rangle^{2s}\langle |\tau_2|-\beta |k_2|\langle \epsilon k_2\rangle\rangle \langle 4\epsilon^2 k_2 p(k)\rangle^{2b_1-}},
\end{align}
where
\begin{equation}\label{eq:5.18}
    p(k) = k^3 - \frac{3k_2}{2}k^2 + \Big(\frac{\alpha + 2\epsilon^2 k_2^2}{2\epsilon^2}\Big)k + \frac{\tau_2 - \alpha k_2^2-\epsilon^2 k_2^4}{4\epsilon^2 k_2}.
\end{equation}

If $\frac{2}{3}|k_2| \leq |k| \leq 2 |k_2|$, then $\frac{\langle k \rangle^{2s}}{\langle k_2 \rangle^{2l}\langle k-k_2\rangle^{2s}}\lesssim \frac{1}{\langle k \rangle^{2l-2s}\langle k-k_2\rangle^{2s}}$. If $|k|\lesssim 1$, $III \lesssim \sigma_2 \leq c_2$ by lemma \ref{l:52}. If $|k|\gg 1$, we argue as in lemma \ref{l:54} to obtain

\begin{equation}
    \left|(2k-k_2) (\alpha + \epsilon^2 (k^2+(k-k_2)^2)) \mp \beta \langle \epsilon k_2\rangle\right|\gtrsim |k|^3,
\end{equation}
from which, we estimate
\begin{equation}
    III \lesssim \sum_k \frac{\max (1,\langle k\rangle^{-2s})}{\langle k \rangle^{2l-2s+4}\langle k_2 p(k)\rangle^{2b_1-}}\lesssim  \sigma_2\leq c_2,
\end{equation}
by lemma \ref{l:52} and $l \geq -2$.

If $|k|\leq \frac{2}{3}|k_2|$ or $|k|\geq 2 |k_2|$, then $\frac{\langle k \rangle^{2s}}{\langle k_2 \rangle^{2l}\langle k-k_2\rangle^{2s}}\lesssim \frac{1}{\langle k_2 \rangle^{2l}}$ since $|k-k_2|\geq \frac{|k|}{2}$. As in lemma \ref{l:53}, if $|k_2|\geq C(\alpha,\beta,\epsilon)$, we have 
\begin{align}
III &\lesssim \sum_{k} \frac{1}{\langle k_2\rangle^{2l+2}\langle k_2 p(k)\rangle^{2b_1-}}\lesssim \sigma_2\leq c_2.
\end{align}
Lastly if $|k_2|\lesssim 1$ (for $|k|\leq \frac{2}{3}|k_2|$) or $|k|\lesssim 1$ (for $|k|\geq 2 |k_2|$), then $\langle k_2 \rangle \simeq 1$ and we have $III\lesssim \sigma_2 \leq c_2$, which concludes the proof of the first inequality of proposition \ref{p:51}. To show the second inequality by the duality argument, it suffices to estimate
\begin{equation}\label{eq:5.21}
    \sum_{k_1+k_2-k=0}\int_{\tau_1+\tau_2-\tau=0} \frac{\langle k \rangle^{l+\rho}\langle k_1 \rangle^{-s}\langle k_2\rangle^{-s}f(\tau_1,k_1)g(-\tau_2,-k_2)w(\tau,k)}{\langle |\tau|-\beta |k|\langle \epsilon k\rangle\rangle^{1/2}\langle \tau_1 + \alpha k_1^2+\epsilon^2k_1^4\rangle^{b_1}\langle \tau_2-\alpha k_2^2-\epsilon^2k_2^4\rangle^{b_2}},
\end{equation}
where $f(\tau,k) = |\hat{u}(\tau,k)|\langle k \rangle^s \langle \tau + \alpha k^2 +\epsilon^2 k^4\rangle^{b_1},\: g(\tau,k) = |\hat{v}(\tau,k)|\langle k \rangle^s \langle \tau + \alpha k^2 +\epsilon^2 k^4\rangle^{b_2}$ and $b_1,b_2 \leq \frac{1}{2}$. Since $\rho>0$, we take $k\neq 0$ in the sum.

\textbf{IV.} $\max\Big(\left||\tau| - \beta |k| \langle \epsilon k\rangle\right|, |\tau_1 + \alpha k_1^2 + \epsilon^2 k_1^4|,|\tau_2 - \alpha k_2^2 - \epsilon^2 k_2^4|\Big) = \left||\tau| - \beta |k| \langle \epsilon k\rangle\right|$.

In this region, the lower bound of the dispersive weight is similar to (\ref{eq:5.17}). For $b_1=b_2=b=\frac{1}{2}-$, it suffices to show
\begin{align}\label{eq:5.27}
    \sup_{\tau,k}IV &\coloneqq \sup_{\tau,k} \frac{\langle k \rangle^{2l+2\rho}}{\langle |\tau|-\beta |k|\langle \epsilon k\rangle\rangle}\sum_{k_1}\int \frac{\langle k_1 \rangle^{-2s}\langle k-k_1\rangle^{-2s} d\tau_1}{\langle \tau_1+\alpha k_1^2+ \epsilon^2 k_1^4\rangle^{2b}\langle \tau - \tau_1 - \alpha (k-k_1)^2 - \epsilon^2(k-k_1)^4\rangle^{2b}}\nonumber\\
    &\lesssim \sup_{\tau,k}\frac{\langle k \rangle^{2l+2\rho}}{\langle |\tau|-\beta |k|\langle \epsilon k\rangle\rangle}\sum_{k_1} \frac{1}{\langle k_1 \rangle^{2s}\langle k-k_1\rangle^{2s}\langle \langle k \rangle p(k_1) \rangle^{4b-1}}<\infty
\end{align}
where $p$ is defined in (\ref{eq:5.18}). If $\frac{2}{5}|k|\leq |k_1|\leq \frac{2}{3}|k|$, then $\frac{|k|}{3}\leq |k-k_1|\leq \frac{5}{3}|k|$ and $\frac{\langle k \rangle^{2l+2\rho}}{\langle k_1 \rangle^{2s}\langle k-k_1\rangle^{2s}}\lesssim \frac{1}{\langle k \rangle^{4s-2l-2\rho}}$. For $|k|\lesssim 1$, (\ref{eq:5.27}) reduces to lemma \ref{l:52}. If $|k|\geq C(\alpha,\beta,\epsilon)$ as in lemma \ref{l:53}, we have 
\begin{equation}
IV \lesssim \sum_{k_1}\frac{1}{\langle k \rangle^{4s-2l-2\rho+2}\langle \langle k \rangle p(k_1)\rangle^{4b-1}}\lesssim \sigma_2,    
\end{equation}
since $4s-2l-2\rho+2\geq 0$. If $\frac{2}{3}|k|\leq |k_1|\leq \frac{3}{2}|k|$ and $|k|\lesssim 1$, then again (\ref{eq:5.27}) reduces to lemma \ref{l:52}. If $|k|\geq C(\alpha,\beta,\epsilon)$, then as in lemma \ref{l:54},
\begin{equation}\label{eq:5.24}
    IV\lesssim \sum_{k_1} \frac{1}{\langle k \rangle^{2s-2l-2\rho+4}\langle  \langle k \rangle p(k_1) \rangle^{4b-1}}\leq \sigma_2,
\end{equation}
since $s-l\geq -2+\rho$. Similarly for $|k_1|\leq \frac{2}{5}|k|$ or $|k_1|\geq \frac{3}{2}|k|$, we treat $|k|\lesssim 1$ and $|k|\geq C(\alpha,\beta,\epsilon)$ separately where for $|k|\geq C(\alpha,\beta,\epsilon)$, we have $\langle |\tau-\beta |k|\langle \epsilon k\rangle|\rangle \gtrsim \langle k \rangle^4$, and therefore we can argue as (\ref{eq:5.24}).

\textbf{V.} $\max\Big(\left||\tau| - \beta |k| \langle \epsilon k \rangle\right|, |\tau_1 + \alpha k_1^2 + \epsilon^2 k_1^4|,|\tau_2 - \alpha k_2^2 - \epsilon^2 k_2^4|\Big) = |\tau_1 + \alpha k_1^2 + \epsilon^2 k_1^4|$.

It suffices to show
\begin{align}
    \sup_{\tau_1,k_1}V &\coloneqq \sup_{\tau_1,k_1} \frac{\langle k_1 \rangle^{-2s}}{\langle \tau_1+\alpha k_1^2+\epsilon^2 k_1^4\rangle}\sum_k \int \frac{\langle k \rangle^{2l+2\rho}\langle k-k_1\rangle^{-2s}d\tau}{\langle |\tau|-\beta |k|\langle \epsilon k \rangle\rangle\langle \tau-\tau_1 - \alpha (k-k_1)^2 -\epsilon^2 (k-k_1)^4 \rangle^{2b_2}}\\
    &\lesssim \sup_{\tau_1,k_1} \frac{\langle k_1 \rangle^{-2s}}{\langle \tau_1+\alpha k_1^2+\epsilon^2 k_1^4\rangle}\sum_{k,\pm} \frac{\langle k \rangle^{2l+2\rho}\langle k-k_1 \rangle^{-2s}}{\langle \epsilon^2 (k-k_1)^4 + \alpha (k-k_1)^2 +\tau_1 \mp \beta k \langle \epsilon k \rangle\rangle^{2b_2-}}
\end{align}
where $|\tau_1 + \alpha k_1^2 + \epsilon^2 k_1^4|\gtrsim |k|\cdot |\alpha (2k_1-k)(\alpha+\epsilon^2(k_1^2+(k-k_1)^2))\mp \beta \langle \epsilon k \rangle|$ and $b_1=\frac{1}{2}$. Note that for $\max\Big(\left||\tau| - \beta |k| \langle \epsilon k \rangle\right|, |\tau_1 + \alpha k_1^2 + \epsilon^2 k_1^4|,|\tau_2 - \alpha k_2^2 - \epsilon^2 k_2^4|\Big) = |\tau_2 - \alpha k_2^2 - \epsilon^2 k_2^4|$, the corresponding $L^\infty L^1$ estimate reduces to the current case by an appropriate change of variable.

If $|k|\leq \frac{|k_1|}{2}$ or $\frac{3|k_1|}{2}\leq |k|\leq \frac{5|k_1|}{2}$, then $|k-k_1|\gtrsim |k_1|$, and therefore $\langle k_1 \rangle^{-2s} \langle k \rangle^{2l+2\rho}\langle k-k_1\rangle^{-2s}\lesssim \langle k_1 \rangle^{-4s}\langle k \rangle^{2l+2\rho}$. Hence $V\lesssim \sigma_1$ if $|k_1|\lesssim 1$, and by lemma \ref{l:53},
\begin{equation}
V\lesssim \sum_{k,\pm} \frac{\langle k_1 \rangle^{-4s}\max (1,\langle k_1 \rangle^{2l+2\rho-2})}{\langle \epsilon^2 (k-k_1)^4 + \alpha (k-k_1)^2 +\tau_1 \mp \beta k \langle \epsilon k \rangle\rangle^{2b_2-}}\lesssim \sigma_1,
\end{equation}
if $|k_1|\geq C(\alpha,\beta,\epsilon)$. If $\frac{|k_1|}{2}\leq |k| \leq \frac{3|k_1|}{2}$, then $\langle k_1 \rangle^{-2s}\langle k \rangle^{2l+2\rho}\langle k-k_1\rangle^{-2s}\lesssim \frac{1}{\langle k \rangle^{2s-2l-2\rho}}$. If $|k|\lesssim 1$, then $V \lesssim \sigma_1$, and by lemma \ref{l:54} if $|k|\geq C(\alpha,\beta,\epsilon)$, then $V \lesssim \sigma_1$ since $2s-2l-2\rho+4 \geq 0$. A similar statement follows for $\frac{5|k_1|}{2}\leq |k|$ if $s-l \geq -2+\rho$ since $|k-k_1|\gtrsim |k|$ and $\langle \tau_1 + \alpha k_1^2 + \epsilon^2 k_1^4 \rangle \gtrsim \langle k \rangle^4$ for sufficiently large $|k|$ by lemma \ref{l:54}.
\end{proof}

We modify the examples of spacetime functions given in \cite{takaoka1999well} to give a converse statement for proposition \ref{p:51}. From the next result, it is deduced that $s \geq -1+ \frac{\rho}{2}$ is necessary for proposition \ref{p:51} to hold. It is of interest to find out whether proposition \ref{p:51} holds for $s \in [-1+\frac{\rho}{2},0)$ when $\rho \in [0,1]$. The proof of the following proposition is presented in the appendix.

\begin{proposition}\label{p:non-sharp}

Suppose $\lVert un \rVert_{X^{s,b-1}_S}\lesssim \lVert u \rVert_{X^{s,b}_S}\lVert n \rVert_{X^{l,b}_W}$ holds for all $u,n \in C^\infty_c (\mathbb{T}\times \mathbb{R})$ for some $s,l,b \in \mathbb{R}$. Then $l \geq \max(2(b-1),-2b)\geq -1$ and $s-l \leq \min(-4(b-1),4b)\leq 2$. Furthermore suppose $\lVert D^\rho (u\overline{v})\rVert_{X^{l,b-1}_W}\lesssim \lVert u \rVert_{X^{s,b}_S}\lVert v \rVert_{X^{s,b}_S}$ holds for all $u,v \in C^\infty_c(\mathbb{T}\times \mathbb{R})$ for some $s,l,b \in \mathbb{R}$, $\rho \in (0,1]$. Then $2s-l-\rho \geq \max(2(b-1),-2b) \geq -1$ and $s-l \geq \max(\rho +4(b-1),\rho-4b)\geq -2+\rho$. 
\end{proposition}

\section{Global well-posedness and semi-classical limit.}\label{gwp}

Here our goal is to extend the local solutions obtained in the previous section. For simplicity, fix $\alpha=\beta=1$. While Guo-Zhang-Guo \cite{guo2013global} used the energy method and a compactness argument to derive global well-posedness results on $\mathbb{R}^d$ for $d=1,2,3$ for initial data with integer Sobolev regularity, we extend their results to the compact domain $\mathbb{T}$ for all initial data in certain fractional Sobolev spaces with improved bounds. We show

\begin{theorem}\label{th:main2}
If $(u_0,n_0,n_1) \in\Omega_G$, then the unique local solution obtained in theorem \ref{th:main1} can be extended to a global solution. More precisely, there exists $(u,n,\partial_t n) \in C_{loc}([0,\infty),H^{s,l})$ that satisfies (\ref{eq:main}) such that for all $T>0$, $(u,n,\partial_t n)$ is a unique solution in $Y^s_{S,T}\times Y^l_{W,T}\times Y^{l-2}_{W,T}$.
\end{theorem}

To fully exploit the conservation law (\ref{conservation}), we assume that $n_0,n_1$ are of mean zero. Recall that both $n$ and $\partial_t n$, assumed to be sufficiently regular, are of mean zero whenever they are defined, and thus $H[u,n,\partial_t n](t) = H[u,n,\partial_t n](0) = H_0<\infty$. In fact, the nonlinear part of energy is bounded above by the linear part by Gagliardo-Nirenberg inequality. Once we establish a global solution for mean zero data, then we invoke the change of variable (\ref{variable}) to conclude that such global extension holds without the mean zero assumption. 

To discuss the semi-classical limit to ZS, we denote $(u^\epsilon,n^\epsilon,\partial_t n^\epsilon)$ by the QZS flow generated by $(u_0^\epsilon,n_0^\epsilon,n_1^\epsilon)$ for $\epsilon \geq 0$. Given a solution $(u^\epsilon,n^\epsilon,\partial_t n^\epsilon)$, we denote $H^\epsilon$ by the corresponding energy and $H^\epsilon_0$ by $H^\epsilon$ at $t=0$. We remark

\begin{remark}\label{semiclassical}
Let $s\geq 4$. If $\sup\limits_{\epsilon}\lVert (u_0^\epsilon,n_0^\epsilon,n_1^\epsilon)\rVert_{H^{s,s-1}}\leq R<\infty$ and $(u_0^0,n_0^0,n_1^0) \in H_0^{s,s-1}$ where $(u_0^\epsilon,n_0^\epsilon,n_1^\epsilon)\xrightarrow[\epsilon\rightarrow 0]{H_0^{s-2,s-3}}(u_0^\epsilon,n_0^\epsilon,n_1^\epsilon)$, then $(u^\epsilon,n^\epsilon,\partial_t n^\epsilon)\xrightarrow[\epsilon \rightarrow 0]{}(u^0,n^0,\partial_t n^0)$ in $C([0,T],H_0^{s-2,s-3})$.  
\end{remark}

It is crucial to obtain a uniform bound on $(u^\epsilon,n^\epsilon,\partial_t n^\epsilon)$ that depends only on $R,T$ (see the next lemma), after which one can adopt the proof of \parencite[theorem 1.3]{guo2013global} to prove remark \ref{semiclassical}.

\begin{lemma}\label{uniformbound}
Let $(s,l)\in \Omega_G$ and $\sup\limits_{\epsilon>0}\lVert (u_0^\epsilon,n_0^\epsilon,n_1^\epsilon)\rVert_{H^{s,l}}\leq R$ for some $R>0$. Then $\sup\limits_{\epsilon>0}\sup\limits_{t\in [0,\infty)}\lVert (u^\epsilon,n^\epsilon,\partial_t n^\epsilon)\rVert_{H_0^{1,0}}\leq C(R)$. If $s \geq 4$, then $\sup\limits_{\epsilon>0}\lVert (u^\epsilon,n^\epsilon,\partial_t n^\epsilon)\rVert_{C_T H_0^{s^\prime,s^\prime-1}}\leq C(T,R)$ for all $1\leq s^\prime \leq s-2$ and $T>0$.
\end{lemma}
We observe that the argument in \cite{guo2013global} in studying the $\epsilon\rightarrow 0$ problem on $\mathbb{R}^d$ applies to $\mathbb{T}$ as well with certain subtleties, which we clarify in the concluding remarks. Now we state a useful Sobolev space inequality:

\begin{lemma}\label{sobinequality}
Let $d \in \mathbb{N},\:s\in [-\frac{d}{2},\frac{d}{2}]$ and consider $H^s(M)$ where $M = \mathbb{R}^d,\mathbb{T}^d$. Then 
\begin{equation*}
\lVert f g \rVert_{H^{s}}\lesssim_{d,s} \lVert f \rVert_{H^{\frac{d}{2}+}}\lVert g \rVert_{H^s}.    
\end{equation*}
\end{lemma}
\begin{proof}
If $s=0$, the statement follows from the H\"older's inequality and the Sobolev embedding $H^{\frac{d}{2}+}\hookrightarrow L^\infty$. If $s<0$, then
\begin{equation}
    \lVert fg \rVert_{H^s} = \sup_{\lVert h \rVert_{H^{-s}}=1} |\langle fg,h\rangle| \leq \lVert g \rVert_{H^s} \sup_{\lVert h \rVert_{H^{-s}}=1} \lVert fh \rVert_{H^{-s}},
\end{equation}
and hence it suffices to show the statement for $s>0$. By the Leibniz's rule,
\begin{align}
    \lVert fg \rVert_{H^s} \lesssim \lVert f \rVert_{W^{s,q}} \lVert g \rVert_{L^r} + \lVert f \rVert_{L^\infty} \lVert g \rVert_{H^s},
\end{align}
where $q\in [2,\infty), r \in (2,\infty]$ are to be determined. The second term is bounded above by $\lVert f \rVert_{H^{\frac{d}{2}+}}\lVert g \rVert_{H^s}$ again by the Sobolev embedding. To obtain
\begin{align}\label{sobemb}
    H^{\frac{d}{2}+}\hookrightarrow W^{s,q},\: H^s \hookrightarrow L^r,
\end{align}
it suffices to have
\begin{align}
    \frac{1}{2}-\frac{1}{q}< \frac{(d/2+)-s}{d},\: \frac{1}{2}-\frac{1}{r}<\frac{s}{d}
\end{align}
and by noting $\frac{1}{2}=\frac{1}{q}+\frac{1}{r}$, we can pick $r \in (2,\infty]$ such that
\begin{align}
\frac{1}{2}-\frac{s}{d}<\frac{1}{r}< \frac{(d/2+)-s}{d}, 
\end{align}
which uniquely determines $q\in [2,\infty)$, and therefore validates (\ref{sobemb}).
\end{proof}

\begin{proof}[proof of theorem \ref{th:main2}]
Assume $(s,l)=(2,1)$. By the Gagliardo-Nirenberg inequality
\begin{equation}
\lVert f \rVert_{L^4}^4 \lesssim \lVert \partial_x f \rVert_{L^2} \lVert f \rVert_{L^2}^3,    
\end{equation}
(\ref{conservation}), and the Young's inequality, we have
\begin{align}\label{nonlinearenergy}
    \left|\int n|u|^2\right| \leq \frac{1}{4}\lVert n \rVert_{L^2}^2 + \frac{\epsilon^2}{2}\lVert \partial_x u \rVert_{L^2}^2 + C(\lVert u_0 \rVert_{L^2},\epsilon),
\end{align}
and since 
\begin{align}\label{bound1}
    \lVert u(t)\rVert_{H^2}^2 + \lVert n(t)\rVert_{H^1}^2 + \lVert \partial_t n(t)\rVert_{H^{-1}}^2 &\lesssim \lVert u_0\rVert_{L^2}^2 + \lVert \partial_{xx} u(t)\rVert_{L^2}^2 + \lVert n(t)\rVert_{L^2}^2 + \lVert \partial_x n \rVert_{L^2}^2 + \lVert \partial_t n(t)\rVert_{\dot{H}^{-1}}^2\\
    &\lesssim \lVert u_0 \rVert_{L^2}^2+(1+\epsilon^{-2})|H_0|+(1+\epsilon^{-2})\left|\int n|u|^2\right|\nonumber,
\end{align}
by using the bound (\ref{nonlinearenergy}) to absorb $\lVert n \rVert_{L^2}, \lVert \partial_x u \rVert_{L^2}$ to the LHS of (\ref{bound1}), we have $\lVert(u,n,\partial_t n) \rVert_{H^{2,1}}\leq C(\epsilon)$ for all $t\in \mathbb{R}$ where $C(\epsilon)\xrightarrow[\epsilon \rightarrow 0]{}\infty$. For $\epsilon>0$, this yields a global solution for $(s,l)=(2,1)$.

Let $(s,l)\in \Omega_G\setminus \left\{(2,1)\right\}$ and denote $l=s-l_0$ for $0\leq l_0 \leq 2$, where since $s+l\geq 4$, we have
\begin{align}\label{admissibleglobal}
s\geq 2+\frac{l_0}{2}\geq 2,\: l \geq 2-\frac{l_0}{2}\geq 1.
\end{align}

With $a= l-2$, multiply $\langle \nabla \rangle^{2a} \partial_t n$ to the second equation of (\ref{eq:main}) and integrate by parts to obtain

\begin{align}\label{energyestwave}
\frac{1}{2}\frac{d}{dt}\Big( \lVert \partial_t n \rVert_{H^a}^2 + \lVert \partial_x n\rVert_{H^a}^2+ \epsilon^2 \lVert \partial_{xx} n\rVert_{H^a}^2\Big) &= \int (\langle \nabla \rangle^a \partial_t n) (\langle \nabla \rangle^a \partial_{xx}|u|^2)\nonumber\\
    &\lesssim_a \lVert \partial_t n\rVert_{H^a}^2 + \lVert u\rVert_{H^{a+2}}^2.    
\end{align}

We first assume $l_0>0$. For $T>0$, we make the following inductive hypothesis:
\begin{align}\label{induct1}
\lVert u \rVert_{H^l}\leq C(T,l_0,\epsilon)<\infty,
\end{align}
from which Gronwall's inequality on (\ref{energyestwave}) yields
\begin{align}
      \lVert \partial_t n \rVert_{H^a}^2 + \lVert \partial_x n\rVert_{H^a}^2+ \epsilon^2 \lVert \partial_{xx} n\rVert_{H^a}^2 \leq C(T),   
\end{align}
which, together with the conservation of energy, controls $\lVert n \rVert_{H^l}$.

Now take $\partial_t$ of the first equation of (\ref{eq:main}), multiply the resulting equation by $-i \langle \nabla \rangle^{2b} \overline{\partial_t u}$, integrate by parts, and take its real part to obtain

\begin{align}\label{energyinequaltiy2}
    \frac{1}{2}\frac{d}{dt}\lVert \partial_t u\rVert_{H^b}^2 &= Im \int \Big( \langle \nabla \rangle^b \overline{\partial_t u}\Big) \Big(\langle \nabla \rangle^b (\partial_t u\cdot n + u\cdot \partial_t n) \Big)\nonumber\\
    &\leq \lVert \partial_t u \rVert_{H^b} (\lVert \partial_t u \cdot n\rVert_{H^b}+\lVert u\cdot \partial_t n\rVert_{H^b}).
\end{align}
We re-write the first equation of (\ref{eq:main}),
\begin{align}
    \Delta u &= \langle \epsilon \nabla \rangle^{-2} (-i\partial_t u + un),
\end{align}
and let $b=s-4$. Further, note that $\lVert \Delta u \rVert_{H^{s-2}}$ controls $\lVert u \rVert_{H^s}$ by mass conservation. We claim
\begin{align}
    \lVert \partial_t u \cdot n\rVert_{H^b}\lesssim \lVert \partial_t u \rVert_{H^b}\lVert n \rVert_{H^l}.
\end{align}
If $s>\frac{9}{2}$, then $b>\frac{1}{2}$ and we have
\begin{align}
    \lVert \partial_t u \cdot n \rVert_{H^b} \lesssim \lVert \partial_t u \rVert_{H^b} \lVert n \rVert_{H^b} \leq \lVert \partial_t u \rVert_{H^b}\lVert n \rVert_{H^{s-4}}.
\end{align}
If $s<\frac{7}{2}$, then $-b > \frac{1}{2}$ and we have
\begin{align}
    \lVert \partial_t u \cdot n \rVert_{H^b} &= \sup_{\lVert \phi \rVert_{H^{-b}}=1} |\langle \partial_t u \cdot n, \phi \rangle| \leq \lVert \partial_t u \rVert_{H^b} \sup_{\lVert \phi \rVert_{H^{-b}}=1} \lVert n \phi \rVert_{H^{-b}} \lesssim \lVert \partial_t u\rVert_{H^b} \lVert n \rVert_{H^{-b}} \leq \lVert \partial_t u\rVert_{H^b} \lVert n \rVert_{H^{l}},
\end{align}
where the last inequality holds since $s \geq 2+\frac{l_0}{2}$.  If $\frac{7}{2}\leq s \leq \frac{9}{2}$, then $-\frac{1}{2}\leq b \leq \frac{1}{2}$ and by lemma \ref{sobinequality}
\begin{align}
    \lVert \partial_t u \cdot n \rVert_{H^b}\lesssim \lVert \partial_t u \rVert_{H^b}\lVert n \rVert_{H^{l}}.
\end{align}
Similarly we have $\lVert u\cdot \partial_t n\rVert_{H^b} \lesssim 1$, and hence
\begin{align}
    \frac{d}{dt}\lVert \partial_t u\rVert_{H^b}^2 \lesssim(\lVert \partial_t u \rVert_{H^b}^2 + \lVert \partial_t u \rVert_{H^b}),
\end{align}
from which Gronwall's inequality yields $\lVert \partial_t u \rVert_{H^{s-4}}\leq C(T)$. Using similar arguments, we have 
\begin{align}
 \lVert un \rVert_{H^{b}}\leq C(T),   
\end{align}
and from
\begin{align}
    \lVert \Delta u \rVert_{H^{s-2}} &\leq \lVert \langle \epsilon \nabla \rangle^{-2} \partial_t u\rVert_{H^{s-2}} + \lVert \langle \epsilon \nabla \rangle^{-2}(un)\rVert_{H^{s-2}}
\end{align}
follows $\lVert u \rVert_{H^s} \leq C(T)$. To show (\ref{induct1}), consider the base case $s_0 = 2+\frac{l_0}{2}$, where by conservation of energy, $\lVert u \rVert_{H^{2-\frac{l_0}{2}}}\leq C$. Then for all $s \in [s_0,s_1]$ where $s_1 = s_0 + l_0$, we have $\lVert u \rVert_{H^{s-l_0}}\leq \lVert u \rVert_{H^{s_0}}$. We iterate this process by an increment of $l_0$ to cover the entire range of $s \geq 2+\frac{l_0}{2}$. It remains to prove the $l_0=0$ case. Let $s\geq 2+\frac{\epsilon_0}{2}$ where $0\leq \epsilon_0 \leq 1$. As before, we consider the energy estimate
\begin{align}\label{energyest2}
    \frac{d}{dt}\Big( \lVert \partial_t n \rVert_{H^{s-2}}^2 + \lVert \partial_x n\rVert_{H^{s-2}}^2+ \epsilon^2 \lVert \partial_{xx} n\rVert_{H^{s-2}}^2\Big)&\lesssim \lVert \partial_t n\rVert_{H^{s-2}}^2+ \lVert u \rVert_{H^{s}}^2\\
    \frac{d}{dt}\lVert \partial_t u\rVert_{H^{s-4}}^2&\lesssim \lVert \partial_t u \rVert_{H^{s-4}}( \lVert \partial_t u \cdot n\rVert_{H^{s-4}}+\lVert u\cdot \partial_t n\rVert_{H^{s-4}})\nonumber,
\end{align}
where by a similar argument as before we derive
\begin{align}
    \lVert \partial_t u \cdot n \rVert_{H^{s-4}} \lesssim \lVert \partial_t u \rVert_{H^{s-4}}\lVert n \rVert_{H^{s-\epsilon_0}};\:\lVert u \partial_t n \rVert_{H^{s-4}}\lesssim \lVert u \rVert_{H^s}\lVert \partial_t n \rVert_{H^{s-4}}.
\end{align}
Recall
\begin{align}
    \lVert \Delta u \rVert_{H^{s-2}} \lesssim \lVert \partial_t u \rVert_{H^{s-4}}+ \lVert un \rVert_{H^{s-4}}\lesssim \lVert \partial_t u \rVert_{H^{s-4}} + \lVert u \rVert_{H^{s-\epsilon_0}}\lVert n \rVert_{H^{s-\epsilon_0}},
\end{align}
and hence
\begin{align}
    \frac{d}{dt}\Big( \lVert \partial_t n \rVert_{H^{s-2}}^2 + \lVert \partial_x n\rVert_{H^{s-2}}^2+ \epsilon^2 \lVert \partial_{xx} n\rVert_{H^{s-2}}^2\Big)&\lesssim \lVert \partial_t n\rVert_{H^{s-2}}^2+\lVert \partial_t u \rVert_{H^{s-4}}^2+\lVert u_0 \rVert_{L^2}^2+ (\lVert u \rVert_{H^{s-\epsilon_0}}\lVert n \rVert_{H^{s-\epsilon_0}})^2\\
    \frac{d}{dt}\lVert \partial_t u\rVert_{H^{s-4}}^2&\lesssim \lVert \partial_t u \rVert_{H^{s-4}}^2 \lVert n \rVert_{H^{s-\epsilon_0}}+\lVert \partial_t u \rVert_{H^{s-4}}(\lVert u_0\rVert_{L^2}+\lVert \partial_t u \rVert_{H^{s-4}} + \lVert u \rVert_{H^{s-\epsilon_0}}\lVert n \rVert_{H^{s-\epsilon_0}})\lVert \partial_t n \rVert_{H^{s-4}}\nonumber.
\end{align}
If $\epsilon_0=0$, then integrate the first differential inequality of (\ref{energyest2}) to obtain an exponential growth bound on $\lVert n \rVert_{H^2} + \lVert \partial_t n \rVert_{L^2}$, and then apply the Gronwall's inequality again to the second differential inequality of (\ref{energyest2}). If $s>2$, use the exponential growth bound for $s=2$ for the base case $s_0 = 2+\frac{\epsilon_0}{2}$. Such exponential bound is obtained for all $s \geq 2+\frac{\epsilon_0}{2}$ by iterating the Gronwall's inequality. Since $\epsilon_0>0$ is arbitrary, we have an exponential bound on the Sobolev norms of solutions for all $s\geq 2$.
\end{proof}

\begin{proof}[proof of lemma \ref{uniformbound}]
By an inspection, $|H_0^\epsilon|\leq C(R)$ uniformly in $\epsilon$. Since mass is conserved and
\begin{align}
    \lVert \partial_x u^\epsilon \rVert_{L^2}^2+\frac{1}{2}\lVert n^\epsilon \rVert_{L^2}^2 +\frac{1}{2}\lVert \partial_t n^\epsilon \rVert_{H^{-1}}&\leq |H_0^\epsilon| + \left|\int n^\epsilon |u^\epsilon|^2\right|\nonumber\\
    &\leq |H_0^\epsilon|+ \frac{1}{4} \lVert n^\epsilon \rVert_{L^2}^2 + \frac{1}{2} \lVert \partial_x u^\epsilon \rVert_{L^2}^2 + C^\prime,
\end{align}
where the last inequality is by the Gagliardo-Nirenberg inequality, and $C^\prime$ is independent of $\epsilon$, we have
\begin{align}\label{energybound}
    \sup_{\epsilon}\sup_{t \in [0,\infty)}\lVert (u^\epsilon,n^\epsilon,\partial_t n^\epsilon)\rVert_{H_0^{1,0}}\leq C(R).
\end{align}
Now assume $s \geq 4, T>0$ and the following inductive hypothesis:
\begin{align}\label{induct2}
    \lVert u^\epsilon \rVert_{H^{s^\prime-2}}, \lVert n^\epsilon \rVert_{H^{s^\prime-2}}, \lVert n^\epsilon \rVert_{H^1} \leq C(T,R),
\end{align}
uniformly in $\epsilon>0$ and $t\in [0,T]$. Energy conservation and (\ref{induct2}) yield $\lVert u^\epsilon n^\epsilon \rVert_{H^{s^\prime-2}} \lesssim C(T,R)$ since $\lVert u^\epsilon n^\epsilon \rVert_{H^{s^\prime-2}} \lesssim \lVert u^\epsilon \rVert_{H^{s^\prime-2}}\lVert n^\epsilon \rVert_{H^{s^\prime-2}} \leq C(T,R)$ for $s^\prime>\frac{5}{2}$ and $\lVert u^\epsilon n^\epsilon \rVert_{H^{s^\prime-2}} \leq \lVert u^\epsilon n^\epsilon \rVert_{H^{1}}\lesssim \lVert u^\epsilon \rVert_{H^1}\lVert n^\epsilon \rVert_{H^1} \leq C(T,R)$ for $s^\prime \in [1,\frac{5}{2}]$. Moreover since $\lVert u^\epsilon \rVert_{\dot{H}^{s^\prime}} \leq \lVert \langle \epsilon \nabla \rangle^2 \Delta u^\epsilon \rVert_{H^{s^\prime-2}}$ for all $\epsilon\geq 0$, we have
\begin{align}
    \lVert u^\epsilon \rVert_{H^{s^\prime}}^2 \lesssim \lVert u_0^\epsilon \rVert_{L^2}^2 + \lVert \partial_t u^\epsilon \rVert_{H^{s^\prime-2}}^2 + \lVert u^\epsilon n^\epsilon \rVert_{H^{s^\prime-2}}^2\leq \lVert u_0^\epsilon \rVert_{L^2}^2 + \lVert \partial_t u^\epsilon \rVert_{H^{s^\prime-2}}^2 + C(T,R),
\end{align}
and hence the differential inequality obtained from the first equation of (\ref{eq:main}) is
\begin{align}\label{gronwall1}
    \frac{d}{dt}\Big(\lVert \partial_t n^\epsilon \rVert_{H^{s^\prime-2}}^2 + \lVert \partial_x n^\epsilon \rVert_{H^{s^\prime-2}}^2+\epsilon^2\lVert \partial_{xx}n^\epsilon \rVert_{H^{s^\prime-2}}^2\Big) \lesssim \lVert \partial_t n^\epsilon \rVert_{H^{s^\prime-2}}^2 + \lVert \partial_t u^\epsilon \rVert_{H^{s^\prime-2}}^2 + C(T,R),
\end{align}
where its LHS is well-defined since $s^\prime \leq s-2 \leq l$ from $(s,l)\in \Omega_G$.

Similarly we derive another differential inequality as in (\ref{energyinequaltiy2}) with $b= s^\prime-2$. A similar calculation as before shows 
\begin{align}
    \lVert \partial_t u^\epsilon \cdot n^\epsilon \rVert_{H^{s^\prime-2}}\lesssim \lVert \partial_t u^\epsilon \rVert_{H^{s^\prime-2}};\:\lVert u^\epsilon \partial_t n^\epsilon \rVert_{H^{s^\prime-2}} \lesssim \lVert \partial_t n^\epsilon\rVert_{H^{s^\prime-2}},
\end{align}
by the inductive hypothesis where the implicit constants are independent of $\epsilon$, and hence by the Young's inequality
\begin{align}\label{gronwall2}
    \frac{d}{dt}\lVert \partial_t u^\epsilon \rVert_{H^{s^\prime-2}}^2 \lesssim \lVert \partial_t u^\epsilon \rVert_{H^{s^\prime-2}}^2 + \lVert \partial_t n^\epsilon \rVert_{H^{s^\prime-2}}^2.
\end{align}
Integrating (\ref{gronwall1}) and (\ref{gronwall2}), we obtain
\begin{align}
\lVert \partial_t n^\epsilon \rVert_{H^{s^\prime-2}}^2 + \lVert \partial_x n^\epsilon \rVert_{H^{s^\prime-2}}^2+\epsilon^2\lVert \partial_{xx}n^\epsilon \rVert_{H^{s^\prime-2}}^2 + \lVert \partial_t u^\epsilon \rVert_{H^{s^\prime-2}}^2 \leq C(T,R).    
\end{align}
Now we check (\ref{induct2}). By \parencite[proposition 2.4]{guo2013global}, we have $\sup\limits_{\epsilon} \lVert n^\epsilon \rVert_{C_TH^1_x}\leq C(T,R)$. On the other hand, if $1\leq s^\prime \leq 2$, then $    \lVert u^\epsilon \rVert_{H^{s^\prime-2}}, \lVert n^\epsilon \rVert_{H^{s^\prime-2}} \leq C$, independent of $\epsilon$, by (\ref{energybound}). Hence for such $s^\prime$
\begin{align}
\sup\limits_{\epsilon>0}\lVert (u^\epsilon,n^\epsilon,\partial_t n^\epsilon)\rVert_{C_T H_0^{s^\prime,s^\prime-1}}\leq C(T,R),    
\end{align}
and using $s_0^\prime=2$ as a base case, we can extend the uniform bound to all $2\leq s^\prime \leq 3$ from which we iterate to cover the entire $1 \leq s^\prime \leq s-2$.
\end{proof}

\begin{remark}
When $T = \infty$, we do not expect continuity as $\epsilon\rightarrow 0$. In fact, the data-to-solution map (where we extend the domain from $D$ to $D\times [0,\infty)$ where $D$ is the data space and $\epsilon \in [0,\infty)$) fails to be continuous at any $\epsilon\geq 0$ as the following example shows. Let $D = H_0^{s,l}$ or $H^{s,l}$ and $(u_0,n_0,n_1) = (\langle N \rangle^{-s}e^{iNx},0,0)$ for $N \in \mathbb{R}\setminus \left\{0\right\}$. Then one can show $(u,n,\partial_t n)(x,t) = (\langle N \rangle^{-s}e^{-it(N^2+\epsilon^2 N^4)+iNx},0,0)$ is the (classical) solution. One can explicitly compute
\begin{equation}
    \sup_{t\in [0,\infty)} \lVert \langle N \rangle^{-s}e^{-it(N^2+\epsilon^2 N^4)+iNx}-\langle N \rangle^{-s}e^{-it(N^2+\epsilon_0^2 N^4)+iNx}\rVert_{H^s_x} = \sup_{t\in [0,\infty)}|1-e^{it(\epsilon^2-\epsilon_0^2)N^4}|=2.
\end{equation}
Note that this example is not valid on $\mathbb{R}^d$.
\end{remark}

\begin{remark}
On $\mathbb{R}^d$, the derivative flow $\partial_t n$ is split into a low and high-frequency; see \cite{guo2013global}. On $\mathbb{T}$, we simply integrate out the mean by (\ref{variable}).
\end{remark}

\begin{remark}
One way to prove remark \ref{semiclassical} is to regularize the intial data, which is done via a particular convolution kernel on $\mathbb{R}^d$ in \cite{guo2013global}, and use the triangle inequality on
\begin{align}
    (u^\epsilon-u^0,n^\epsilon-n^0,\partial_t n^\epsilon-\partial_t n^0) =(u^\epsilon-u^{0,h},n^\epsilon-n^{0,h},\partial_t n^\epsilon-\partial_t n^{0,h})+(u^{0,h}-u^0,n^{0,h}-n^0,\partial_t n^{0,h}-\partial_t n^0), 
\end{align}
where on a periodic domain, one can define a family of mollifiers as a Fourier multiplier as follows: let $\eta \in C_c^\infty (\mathbb{R})$ that is identically one in the neighborhood of the origin. For $h>0$, define $\widehat{J_h f} (k) =\eta(hk)\widehat{f}(k)$ for all $f \in L^1(\mathbb{T})$. Then $\lVert J_h f - f \rVert_{H^s} \xrightarrow[h\rightarrow 0]{}0$ and for $\sigma>0$
\begin{align}
    \lVert J_h f - f \rVert_{H^{s-\sigma}} \leq C(\sigma)h^\sigma\lVert f \rVert_{H^s};\: \lVert J_h f \rVert_{H^{s+\sigma}} \leq \frac{C(\sigma)}{h^\sigma} \lVert f \rVert_{H^s}.
\end{align}
\end{remark}
\section{Appendix.}\label{appendix}
\begin{proof}[proof of proposition \ref{p:non-sharp}]
The main idea is to add a fourth-order perturbation to the spacetime functions constructed in \cite{takaoka1999well}. Once those examples are given, one can directly substitute the examples to the inequalities in proposition \ref{p:non-sharp} to derive a set of necessary conditions on the scaling parameter $N\gg 1$. Let $\delta(k)$ be the Kronecker delta function defined on $\mathbb{Z}$ and let $\phi(\tau)$ be a smooth bump function on $\mathbb{R}$ with a compact support. It suffices to consider $u_i,\:1\leq i \leq 8$, $n_i,\:1\leq i \leq 4$, and $v_i,\:5 \leq i \leq 8$ where
\begin{align*}
    \widehat{u}_1(k,\tau) = \delta(k+N)\phi (\tau+\alpha N^2 + \epsilon^2 N^4)&;\: \widehat{n}_1(k,\tau) = \delta(k-2N)\phi(|\tau|-2\beta N\langle 2\epsilon N\rangle),\\
    \widehat{u}_2(k,\tau) = \delta(k+N)\phi (\tau+\alpha N^2 + \epsilon^2 N^4+2\beta N \langle 2\epsilon N\rangle)&;\: \widehat{n}_2(k,\tau) = \delta(k-2N)\phi(|\tau|-2\beta N\langle 2\epsilon N\rangle),\\
    \widehat{u}_3(k,\tau) = \delta(k)\phi (\tau)&;\: \widehat{n}_3(k,\tau) = \delta(k-N)\phi(|\tau|-\beta N\langle \epsilon N\rangle),\\
    \widehat{u}_4(k,\tau) = \delta(k)\phi (\tau+\alpha N^2 + \epsilon^2 N^4+\beta N \langle \epsilon N\rangle)&; \:\widehat{n}_4(k,\tau)= \delta(k-N)\phi(|\tau|-\beta N\langle \epsilon N\rangle),\\
    \widehat{u}_5(k,\tau) = \delta(k-N)\phi (\tau+\alpha N^2 + \epsilon^2 N^4)&;\:\widehat{v}_5(k,\tau)= \delta(k+N)\phi (\tau+\alpha N^2 + \epsilon^2 N^4),\\
    \widehat{u}_6(k,\tau) = \delta(k-N)\phi (\tau+\alpha N^2 + \epsilon^2 N^4-2\beta N \langle 2\epsilon N\rangle)&;\:\widehat{v}_6(k,\tau)= \delta(k+N)\phi (\tau+\alpha N^2 + \epsilon^2 N^4),\\
    \widehat{u}_7(k,\tau) = \delta(k)\phi (\tau)&;\:\widehat{v}_7(k,\tau)= \delta(k-N)\phi (\tau+\alpha N^2 + \epsilon^2 N^4),\\
    \widehat{u}_8(k,\tau) = \delta(k)\phi (\tau+\alpha N^2 + \epsilon^2 N^4+\beta N \langle \epsilon N\rangle)&;\:\widehat{v}_8(k,\tau)= \delta(k-N)\phi (\tau+\alpha N^2 + \epsilon^2 N^4).
\end{align*}
\end{proof}

\printbibliography

\end{document}